\documentclass[11pt]{article}
\usepackage{amsfonts}
\usepackage{amsmath}
\usepackage{amsthm}

\usepackage{tikz}
\newcommand*\circled[1]{\tikz[baseline=(char.base)]{
    \node[shape=circle,draw,inner sep=2pt] (char) {#1};}}

\makeatletter
\renewcommand*\env@matrix[1][*\c@MaxMatrixCols c]{%
  \hskip -\arraycolsep
  \let\@ifnextchar\new@ifnextchar
  \array{#1}}
\makeatother

\usepackage{amssymb}
\usepackage{enumerate}
\usepackage[latin1]{inputenc} 
\usepackage{color}
\usepackage{pifont}
\usepackage{multirow}
\usepackage[
paperwidth=210mm,paperheight=297mm, textwidth=170mm,textheight=260mm, lmargin=20mm,top=10mm,
]{geometry}
\usepackage{algorithm}
\usepackage{algorithmic}
\usepackage{mathdots}

\DeclareMathOperator \rank {rank}
\DeclareMathOperator \Mrank {Mrank}
\DeclareMathOperator \mrank {mrank}

\usepackage{arydshln,leftidx,mathtools}

\newtheorem{theorem}{Theorem}[section]
\newtheorem*{theoremWOnumbering2.4}{Theorem 2.4 (detailed version)}
\newtheorem*{theoremWOnumbering2.5}{Theorem 2.5 (detailed version)}
\newtheorem{proposition}[theorem]{Proposition}

\newtheorem{note}[theorem]{Note}
\newtheorem{remark}[theorem]{Remark}
\newtheorem*{remarkWOnumbering}{Remark}

\newtheorem{definition}[theorem]{Definition}
\newtheorem{corollary}[theorem]{Corollary}
\newtheorem{example}[theorem]{Example}
\newtheorem{lemma}[theorem]{Lemma}

\newenvironment{smat}{\left[\begin{smallmatrix}}{\end{smallmatrix}\right]}

\title{ACI-matrices of constant rank over arbitrary fields
   }
\author{Alberto Borobia\footnote{Supported by the Spanish Ministerio de Ciencia y Tecnología MTM2015-68805-REDT} , Roberto Canogar\\
\small Universidad Nacional de Educaci\'on a Distancia (UNED), 28040 Madrid, Spain\\
\small e-mail: $aborobia@mat.uned.es$, $rcanogar@mat.uned.es$ }
\date{}

\begin{document}

\maketitle

\begin{abstract}
The columns of a $m\times n$ ACI-matrix over a field $\mathbb{F}$  are independent affine subspaces of $\mathbb{F}^m$. An ACI-matrix has constant rank $\rho$ if all its completions have rank $\rho$. Huang and Zhan (2011) characterized the $m\times n$ ACI-matrices of constant rank when $|\mathbb{F}|\geq \min\{m,n+1\}$. We  complete their result characterizing the $m\times n$ ACI-matrices of constant rank over arbitrary fields. Quinlan and McTigue (2014) proved that every partial matrix of constant rank $\rho$ has a $\rho\times \rho$ submatrix of constant rank $\rho$ if and only $|\mathbb{F}|\geq \rho$. We  obtain an analogous result for ACI-matrices over arbitrary fields by introducing the concept of complete irreducibility.
 \end{abstract}

\section{Introduction}

 Let $\mathbb{F}[x_1, \ldots , x_k]$ denote  the set of polynomials in the   indeterminates $x_1,  \ldots , x_k$ with coefficients on a field $\mathbb{F}$. A matrix  over $\mathbb{F}[x_1, \ldots, x_k]$ is  an \textbf{A}ffine \textbf{C}olumn \textbf{I}ndependent matrix or {\bf ACI-matrix} if its entries are polynomials of degree at most one and  no indeterminate appears in two different columns.  A completion of an ACI-matrix is an assignment  of values in $\mathbb{F}$ to the indeterminates $x_1,\ldots,x_k$.  The ACI-matrices where introduced in 2010 by Brualdi, Huang and Zhan~\cite{MR2680270} as a generalization of  {\bf partial matrices} (matrices whose entries are either a constant or an indeterminate and with each indeterminate only appearing  once). They proposed in~\cite[Problem 5]{MR2680270} the problem  of determining those $m\times n$ ACI-matrices  such that the rank of any of  its completions  is equal to $\rho$ with $0 \leq \rho\leq \min \{m, n\}$.

\subsection{A geometric interpretation}

Let us consider a collection $\mathfrak{C}$ of $n+1$ affine  subspaces of $\mathbb{F}^m$ where  $\mathbb{F}$ is a field.   If we choose one point of each one of the $n+1$ affine subspaces of $\mathfrak{C}$ then the dimension of the affine subspace spanned by these $n+1$ points is an integer of the set $\{0,1,\ldots,\min \{m, n\}\}$. An interesting problem is to determine  for any  $\rho\in \{0,1,\ldots,\min \{m, n\}\}$  how are those collections $\mathfrak{C}=\{\mathcal{V}_0,\mathcal{V}_1,\ldots,\mathcal{V}_n\}$ such that
$$
\{\dim \langle P_0, P_1,\ldots, P_n\rangle : P_i\in \mathcal{V}_i \text{ for } i=0,1,\ldots,n\}=\{\rho\}.
$$
As we will see below this  question  for the particular case in which $\mathcal{V}_0$  is the origin of $\mathbb{F}^m$ coincides with the problem proposed by Brualdi, Huang and Zhan.

Let $\mathcal{V}_1, \ldots, \mathcal{V}_n$ be $n$ affine subspaces of $\mathbb{F}^m$.  If   $\mathcal{V}_j$  has  dimension $d_j$ then it can be parametrized, with respect to the canonical base of $\mathbb{F}^m$, by
\begin{align*}
\begin{bmatrix}
c^{(j)}_1 \\
 \vdots \\
 c^{(j)}_m
\end{bmatrix}+
x^{(j)}_1 \begin{bmatrix}
a^{(j)}_{11}  \\
\vdots  \\
a^{(j)}_{m1} 
\end{bmatrix}+\cdots+x^{(j)}_{d_j} \begin{bmatrix}
a^{(j)}_{1d_j}  \\
\vdots  \\
a^{(j)}_{md_j} 
\end{bmatrix}=
\begin{bmatrix}
c^{(j)}_1+\sum_{k=1}^{d_j} a^{(j)}_{1k} x^{(j)}_k \\
\vdots \\
c^{(j)}_m+\sum_{k=1}^{d_j}  a^{(j)}_{mk} x^{(j)}_k 
\end{bmatrix}.
\end{align*}
So, it seems quite natural to represent  the collection $\{\mathcal{V}_1, \ldots, \mathcal{V}_n\}$ by the $m\times n$ ACI-matrix
\begin{align} \label{genericACI}
A=\begin{bmatrix}[ccc]
c^{(1)}_1+\sum_{k=1}^{d_1} a^{(1)}_{1k} x^{(1)}_k  & \quad \cdots \quad &  c^{(n)}_1+\sum_{k=1}^{d_n} a^{(n)}_{1k} x^{(n)}_k  \\
\vdots & \ddots & \vdots \\
c^{(1)}_m+\sum_{k=1}^{d_1} a^{(1)}_{mk} x^{(1)}_k &\quad  \cdots \quad & c^{(n)}_m+ \sum_{k=1}^{d_n} a^{(n)}_{mk} x^{(n)}_k
\end{bmatrix}
\end{align}
where the column $j$ corresponds  to the affine subspace $\mathcal{V}_j$.  

A completion $\widehat{A}$ of the ACI-matrix $A$ given in~(\ref{genericACI}) is an  assignment  of  values in  $\mathbb{F}$ to each one of the indeterminates 
$$x^{(1)}_1,\ldots,x^{(1)}_{d_1}; \ldots\ldots;x^{(n)}_{1},\ldots,x^{(n)}_{d_n}.$$ 
Observe that the column $j$ of  $\widehat{A}$  corresponds to a point   $P_j\in \mathcal{V}_j$. Therefore if    $P_0=(0,\ldots,0)$ is the origin of $\mathbb{F}^m$ then
$$\rank (\widehat{A})=\dim \langle \overrightarrow{P_0 P_1},\ldots, \overrightarrow{P_0 P_n}\rangle=\dim \langle P_0, P_1,\ldots, P_n\rangle.$$
For any    $\rho\in \{0,1,\ldots,\min \{m, n\}\}$ the problem of determining those collections $\{ P_0, \mathcal{V}_1, \ldots, \mathcal{V}_n\}$ of affine subspaces  of $\mathbb{F}^m$ such that   $\dim\langle P_0, P_1,\ldots, P_n\rangle=\rho$ for any choice of points $P_j \in \mathcal{V}_j$ for $j=0,1,\ldots,n$  coincides  with the problem of determining those $m\times n$ ACI-matrices over $\mathbb{F}$  such that $\rho$ is the rank of any of  its completions.

\subsection{The rank of an ACI-matrix}

\begin{definition} Let $A$ be a $m\times n$ ACI-matrix over $\mathbb{F}$. The {\bf rank of} $A$, $\rank(A)$, is the set  of integers that are the rank of some  completion of $A$. The {\bf Mrank of} $A$, $\Mrank(A)$, is the highest rank of a completion of $A$, and the  {\bf mrank of} $A$, $\mrank(A)$, is the lowest rank of a completion of $A$.  We say that   $A$ has {\bf constant rank} $\rho$ if $\Mrank(A)=\mrank(A)=\rho$, that is, if $\rank(A)=\{\rho\}$.  
\end{definition}

An {\bf A}ffine {\bf C}olumn or \mbox{\bf A-column} of size $m$ is an ACI-matrix with  one column and $m$ rows. The ACI-matrices are described in terms of independent A-columns, where independent means that the A-columns share no variables. So  $\big[\, C_1\ \cdots \ C_n \,\big]$ is an $m\times n$ ACI-matrix if and only if $C_1,\ldots,C_n$ are  independent A-columns of size $m$. 
The use of A-columns  help us to introduce several  concepts that  appear when we consider  ACI-matrices of constant  rank. 

 \begin{definition} \label{mfr}
Let  $A=\big[\, C_1\ \cdots \ C_n \,\big]$ be an $m\times n$  ACI-matrix over $\mathbb{F}$ of  constant rank $\rho$.  We say that $A$ is  \textbf{full rank} if  $\rho=\min \{m, n\}$. We distinguish three special types of full rank ACI-matrices:
 \begin{itemize}

\item   $A$ is  {\bf square full rank} if  $\rho=n=m$.

\item   $A$ is {\bf minimal full rank} if  $\rho=m<n$   and  for each  $j\in\{1,\ldots,n\}$ the $m\times(n-1)$ ACI-matrix $$\big[\, C_1\ \cdots \ C_{j-1} \ C_{j+1} \ \cdots \ C_n \,\big]$$
 is not full rank (i.e., it is not of constant rank $\rho=m\leq n-1$).

\item  $A$ is  {\bf maximal full rank} if $\rho=n<m$  and  for each   $v\in \mathbb{F}^m$ the $m\times(n+1)$ ACI-matrix $$\big[\, C_1\ \cdots \ C_n \ v\, \big]$$ is not full rank (i.e., it is not of constant rank $\rho+1=n+1\leq m$).

\end{itemize}

\end{definition}

It is important to keep in mind that if $A$ is  minimal full rank  then it  has less rows than columns, and that if $A$ is  maximal full rank  then it has more  rows   than  columns.

\begin{example}\label{mfrMFR}
In~\cite{BoCa2} we  showed  that there exist minimal and maximal  full rank ACI-matrices over all finite fields. Namely, let $\mathbb{F}_q=\{f_1,\ldots,f_q\}$ be the field  with $q$ elements:
\begin{enumerate}[(i)]

\item  \cite[Example 3.1]{BoCa2}  The following  $2 \times (q+1)$  ACI-matrix over $\mathbb{F}_q$ is minimal full rank:
\begin{equation*} \label{basic-minimal} 
\left[\begin{array}{ccccc}
1+f_1x_1     & \cdots  & 1+f_q x_q & x_{q+1}  \\
x_1 &            \cdots & x_q &1 
\end{array}\right].
\end{equation*}

\item  \cite[Proposition 2.1]{BoCa2} The following  $2q \times (q+1)$  ACI-matrix over $\mathbb{F}_q$ is maximal full rank:
\begin{equation*}\label{basic-maximal}
\left[\begin{array}{ccc|c}
1 & \cdots  & 0 & x_{q+1}-f_1   \\
\vdots & \ddots  & \vdots &  \vdots     \\
0 & \cdots   & 1 & x_{q+1}-f_q \\ \hline
x_1 & \cdots  & 0 & 1 \\
\vdots & \ddots  & \vdots & \vdots  \\
0 & \cdots  & x_q & 1  \\ 
\end{array}\right].
\end{equation*}
\end{enumerate}
On the other hand, in~\cite[Corollaries 2.1 and 3.1]{BoCa2} we showed that  minimal or maximal  full rank ACI-matrices over infinite fields do not exist.

\end{example}

\begin{example}\label{examplemfrMFR} 
\begin{enumerate}[(i)]
\item  How to check that a given ACI-matrix is minimal full rank? Consider the $3\times 5$ ACI-matrix over $\mathbb{F}_2$
\[
A=\left[
\begin{array}{ccccc}
   1 & y_2 & y_3 & 0 & 0 \\
  0 & 0 & y_3 & y_4 & 1 \\
  y_1 & 1 & 1 & 1 & y_5 \\
\end{array}
\right].
\]
Note  that $A$  is full rank  since it has 5 variables and admits $2^5$ different completions, all of them of rank 3. Moreover,   for each  $i=1,2,3,4,5$ if we delete the $i^{th}-$column  the resulting $3\times 4$ ACI-matrix is not full rank since it admits a completion of rank 2.  So $A$ is minimal full rank. 

%\item Now we  point at a sensitive property of the definition of maximal full rank ACI-matrices.  Let  $A=\begin{smat}  x \\ 1+x \end{smat}$ be a  $2 \times 1$ ACI-matrix    over $\mathbb{F}_2$. Note that $A$  has constant rank 1 since its  two completions have rank equal to 1. Let $\{\begin{smat}  1 \\ 0 \end{smat},\begin{smat}  0 \\ 1 \end{smat}\}$ be the canonical base of the vectorial space $\mathbb{F}^2_2$.  It can be easily checked that 
%$\rank\begin{smat}  x & 1 \\ 1+x & 0 \end{smat}=\rank\begin{smat}  x & 0 \\ 1+x & 1 \end{smat}=\{1,2\}$.
%So, none of the two augmented ACI-matrices is full rank. Does it  imply that $A$ is maximal? No,  since $\rank\begin{smat}  x & 1 \\ 1+x & 1 \end{smat}=\{2\}$.

\item Now we  point at a sensitive property of the definition of maximal full rank ACI-matrices.  The $5 \times 3$ ACI-matrix  over $\mathbb{F}_2$ 
$$A=\begin{bmatrix}  x_1 & 1 & 1 \\  1 & 0 & x_4 \\ x_1 & 0 & x_4 \\ x_2 & 0 & 1 \\ 0 & x_3 & 1 \end{bmatrix}$$
has constant rank 3 since all its  $2^4$ completions have rank equal to 3. Let $\{e_1,e_2,e_3,e_4,e_5\}$ be the canonical base of the vectorial space $\mathbb{F}^5_2$.  It can be checked that 
$$\rank\begin{bmatrix}  A & e_1 \end{bmatrix}=\rank\begin{bmatrix}  A & e_2 \end{bmatrix}=\rank\begin{bmatrix}  A & e_3 \end{bmatrix}=\rank\begin{bmatrix}  A & e_4 \end{bmatrix}=\rank\begin{bmatrix}  A & e_5 \end{bmatrix}=\{3,4\}.$$
None of these five  augmented ACI-matrices is full rank. Does this  imply that $A$ is maximal full rank? No,  since 
$$\rank \begin{bmatrix}  x_1 & 1 & 1 & 1\\  1 & 0 & x_4 & 0 \\ x_1 & 0 & x_4 & 1 \\ x_2 & 0 & 1 & 1 \\ 0 & x_3 & 1 & 0\end{bmatrix}=\{4\}.$$
In linear algebra it is usually enough to check a property for a basis to conclude that this property is true for all vectors. Although this is not the case when one wants to check that an ACI-matrix is maximal full rank.
%So, if $\mathbb{F}_q$ is a field of $q$ elements,  $\{e_1,\ldots,e_q\}$ is a basis of $\mathbb{F}_q$, and $A$ is a full rank ACI-matrix over $\mathbb{F}_q$ then it is not possible to conclude that  $A$ is maximal full rank by only checking that $[A \ e_i]$ is not maximal full rank for $i=1,\ldots,q$. 
\end{enumerate}

\end{example}

\subsection{Equivalent ACI-matrices} \label{equivACIs}

Assume that in the ACI-matrix $A=\big[\, C_1\ \cdots \ C_n \,\big]$ the A-columns $C_1,\ldots,C_n$ are parametrized with respect to the canonical base of $\mathbb{F}^m$. If we consider a different base of $\mathbb{F}^m$ then the parametrization of  $C_1,\ldots,C_n$ with respect to this new base changes, although geometrically  $C_1,\ldots,C_n$ do not change. This  new parametrization  is obtained by multiplying $A$ from the left by a nonsingular constant matrix of order $m$.  Note also  that the order of the columns of $A$ has no impact on its rank.  These two observations motivate us to introduce in a natural way the  terminology of equivalent ACI-matrices.

 \begin{definition} 
Two ACI-matrices   $A$ and $B$ of the same size $m\times n$ are {\bf equivalent}, $A \sim B$, if there exist a nonsingular constant  $T$ of order $m$ and a permutation $Q$ of order $n$ such that  $TAQ=B$. 
\end{definition}
The use of the permutation $Q$ in the definition is  not essential,  but it is useful. It permits to reorganize the columns of an ACI-matrix so that  its structure  becomes more apparent. 
\begin{remark} \label{rowEchelonForm}
Given an $m\times n$ constant matrix $A$ of rank $\rho$, it is well know that there exists a nonsingular constant $T$ of order $m$ such  that $TA$ is in row reduced echelon form. This is know as the Gauss elimination method. Moreover, there exists a permutation $Q$ of order $n$ such that 
$$TAQ=\scriptsize{\left[\begin{array}{ccc|cccc}
 1 &  & 0 &   \multicolumn{1}{c}{\multirow{3}{*}{\Large $\ * \ $}}   \\ 
  & \ddots & &    \\ 
 0 & & 1 &   \\ \hline
 \multicolumn{3}{c|}{\multirow{2}{*}{\Large $\ 0 \ $}} &  \multicolumn{1}{c}{\multirow{2}{*}{\Large $\ 0 \ $}} \\
 &&&
\end{array}\right]}
$$
where we group together the $\rho$ columns corresponding to the pivots in the first $\rho$ columns. The blocks on the right do not appear if $\rho=n$ and the blocks on the bottom do not appear if $\rho=m$. So we have found a representative for the equivalence class of $A$ with a simple structure. The equivalence for ACI-matrices of constant rank is, of course, an extension of the equivalence for constant matrices, and it is  introduced with the idea of finding a representative with a simple structure that reveals its rank. Obviously  the rank of an ACI-matrix is preserved by equivalence.
\end{remark}

 In our following  result we will  see  that also minimality  and  maximality  are preserved by equivalence.

\begin{lemma} \label{MaxAndMin}
Let $A$ and $B$ be equivalent ACI-matrices. We have that:
\begin{enumerate}[(i)]
\item $A$ is   minimal full rank  if and only if $B$ is minimal full rank.
\item $A$ is   maximal full rank  if and only if $B$ is maximal full rank.
\end{enumerate}
\end{lemma}

\begin{proof}  
Let   $m\times n$ be the size of $A$ and $B$. As  $A \sim B$ then  there exist a nonsingular constant  $T$ of order $m$ and a permutation $Q$ of order $n$ such that $B=TAQ$.  
\begin{enumerate}[$(i)$]
\item  Assume that $A$ is minimal full rank. Then $B$ is full rank since 
\begin{align*}
\rank(B) &=\rank(TAQ)=\{\rank(T \widehat{A}Q) : \widehat{A} \text{ completion of $A$}\}\\ &=\{\rank(\widehat{A}) : \widehat{A} \text{ completion of $A$}\}=\rank(A)=\{m\}.
\end{align*}

Let us see now that $B$ is minimal full rank. First we introduce some useful notation:  $C_k(H)$ will denote the ACI-matrix obtained by deleting the column $k$ of the ACI-matrix $H$. 

As $B=TAQ$ then  the columns of $B$ are obtained by permuting the columns of $TA$.    Let $\sigma$ be a permutation of $\{1,\ldots,n\}$  such that  for each $j\in \{1,\ldots,n\}$  the column $j$ of $B$ is equal  to the column $\sigma(j)$ of $TA$.     So  $B$ is  minimal full rank since
\begin{equation*} \label{Bminimalfullrank}
rank\big(C_j(B)\big)=rank \big(C_{\sigma(j)}(TA)\big)=rank \big(TC_{\sigma(j)}(A)\big)=rank  \big(C_{\sigma(j)}(A)\big)\neq\{m\}
\end{equation*}
where the last inequality  follows from the fact that    $A$ is  minimal full rank.  

\item Assume that $A$ is maximal full rank. Then $B$ is full rank since, as in item $(i)$,  
$$\rank(B)=\rank(TAQ)=\rank(A)=\{n\}.$$
We conclude that  $B$ is maximal full rank since  for each $v\in \mathbb{F}^m$ we have
$$rank \Big( \big[B \ |\  v\big] \Big)=rank \Big( T^{-1}\big[ B \ |\  v \big]\Big)=rank \Big( \big[ T^{-1}B \ |\  T^{-1}v \big]\Big)=  $$ $$= rank  \Big( \big[AQ \ |\  T^{-1}v \big]\Big)=rank \Big( \big[ A \ |\  T^{-1}v \big]\Big) \neq \{n+1\}.$$
where the last inequality follows from the fact that    $A$ is  maximal full rank.  
\end{enumerate}

\end{proof}

\section{ACI-matrices of constant rank over arbitrary fields}

We start with a basic result that  will be employed several times in this work.

\begin{lemma} \label{SC+FRC=constant-rank}
Consider an ACI-matrix  $\begin{smat} A_{11} & A_{12} \\  0_{r\times s} & A_{22}\end{smat}$   where $r>m$, $s>n$ and  $r+s\geq \max\{m,n\}$.   The following two statements are equivalent:
\begin{enumerate}[(i)]
\item  $\rank\begin{smat} A_{11} & A_{12} \\  0_{r\times s} & A_{22}\end{smat}=\{(m-r)+(n-s)\}$.
\item $\rank(A_{11})=\{m-r\}$ and $\rank(A_{22})=\{n-s\}$.
\end{enumerate}
\end{lemma}
\begin{proof}  Observe that $A_{11}$ is $(m-r)\times s$ with $m-r\leq s$ and that  $A_{22}$ is $r\times (n-s)$ with $n-s\leq r$. Let $\begin{smat} \widehat{A_{11}} & \widehat{A_{12}} \\  0_{r\times s} & \widehat{A_{22}}\end{smat}$  be any completion of $\begin{smat} A_{11} & A_{12} \\  0_{r\times s} & A_{22}\end{smat}$.

$(i) \Rightarrow (ii)$  We have that 
$${\small (m-r)+(n-s)
=\rank\begin{bmatrix} \widehat{A_{11}} & \widehat{A_{12}} \\  0_{r\times s} & \widehat{A_{22}} \end{bmatrix}
\leq \rank(\widehat{A_{11}})+\rank(\begin{bmatrix}  \widehat{A_{12}} \\  \widehat{{A}_{22}}  \end{bmatrix})\leq (m-r)+(n-s)}$$

then $\rank(\widehat{A_{11}})=m-r$ and so $A_{11}$ has constant rank $m-r$.
$${\small (m-r)+(n-s)=\rank\begin{bmatrix} \widehat{A_{11}} & \widehat{A_{12}} \\  0_{r\times s} & \widehat{A_{22}} \end{bmatrix}\leq \rank\begin{bmatrix} \widehat{A_{11}} & \widehat{A_{12}}  \end{bmatrix}+\rank(\begin{bmatrix}    \widehat{{A}_{22}}  \end{bmatrix})\leq (m-r)+(n-s)}$$

then $\rank(\widehat{A_{22}})=n-s$ and $A_{22}$ has constant rank $n-s$.

$(ii) \Rightarrow (i)$ The rank of $\widehat{A_{22}}$ is equal to the  number of columns of $\begin{smat}  \widehat{A_{12}} \\  \widehat{{A}_{22}}  \end{smat}$. Then we have  
$${\small \rank\begin{bmatrix} \widehat{A_{11}} & \widehat{A_{12}} \\  0_{r\times s} & \widehat{A_{22}} \end{bmatrix}=\rank\begin{bmatrix} \widehat{A_{11}} & 0 \\  0_{r\times s} & \widehat{A_{22}} \end{bmatrix}=\rank(\widehat{A_{11}})+\rank(\widehat{A_{22}})=(m-r)+(n-s).}$$

And so $\rank\begin{smat} A_{11} & A_{12} \\  0_{r\times s} & A_{22}\end{smat}=\{(m-n)+(n-s)\}$.

\end{proof}

For a better understanding of  the structure of the constant rank ACI-matrices we will make use of the following result of Brualdi, Huang and Zhan~\cite[Theorem 3]{MR2680270}.

\begin{theorem}(\cite{MR2680270}) \label{SubmatrizCeros}
Let $A$ be an $m \times n$  ACI-matrix over an arbitrary field $\mathbb{F}$ and let $\rho$ be an integer such that $1 \leq \rho < \min \{m, n\}$.  The following two statements are equivalent:
\begin{enumerate}[(i)]
\item $\Mrank(A)\leq \rho$.
\item  For some positive integers  $r$ and $s$  with  $\rho=(m-r)+(n-s)$ there exist a nonsingular constant $T$ of order $m$ and a permutation  $Q$ of order $n$ such that  $TAQ= \small{\left[\begin{array}{cc}A_{11} & A_{12} \\  0_{r\times s} & A_{22} \end{array}\right]}$.   The upper blocks $A_{11}$ and $A_{12}$ do not appear if $r=m$ and the right blocks  $A_{21}$ and $A_{22}$ do not appear if $s=n$. 
\end{enumerate}
\end{theorem}

\bigskip
Observe that   $\rho=(m-r)+(n-s)$ and $\rho < \min \{m, n\}$ implies that 
$$r+s=m+n-\rho> \max\{m,n\} $$
%and also that 
%\begin{align*} \label{rho=m-r+n-s}
%m-r<(m-r)+(n-\rho) =s \qquad \text{and} \qquad  n-s< (n-s)+(m-\rho)=r
%\end{align*}
So $m-r<s$ and  and $n-s<r$. Therefore, in part $(ii)$ of Theorem~\ref{SubmatrizCeros}, $A_{11}$ has less rows than columns   and   $A_{22}$  has less columns than rows.   With all this in mind, an immediate consequence of Theorem~\ref{SubmatrizCeros} and Lemma~\ref{SC+FRC=constant-rank} is the following result for    ACI-matrices of constant rank.

\pagebreak

\begin{corollary} \label{SubmatrizCeros3}
Let $A$ be an $m \times n$  ACI-matrix over an arbitrary field $\mathbb{F}$ and let $\rho$ be an integer such that $1\leq  \rho < \min \{m, n\}$.  The following two statements are equivalent:
\begin{enumerate}[(i)]
\item $\rank(A)=\{\rho\}$.
\item For some positive integers  $r$ and $s$  with  $\rho=(m-r)+(n-s)$ there exist a nonsingular constant $T$ of order $m$ and a permutation  $Q$ of order $n$ such that   $TAQ= \small{\left[\begin{array}{cc}A_{11} & A_{12} \\  0_{r\times s} & A_{22} \end{array}\right]}$. The upper blocks  $A_{11}$ and $A_{12}$  do not appear if $r=m$ and the right blocks  $A_{21}$ and $A_{22}$ do not appear if $s=n$. Moreover,  if $r<m$  then  $\rank(A_{11})=\{m-r\}$   and    if  $s<n$ then $\rank(A_{22})=\{n-s\}$.
\end{enumerate}
\end{corollary}

\subsection{Sufficient and necessary condition for ACI-matrices of constant rank}

Huang and Zhan in~\cite[Theorem 5]{MR2775784} characterized the  $m\times n$ ACI-matrices of constant rank over a field  $\mathbb{F}$  with $|\mathbb{F}|\geq \max\{m,n+1\}$.

\begin{theorem}(\cite{MR2775784})\label{ResultadosAnteriores0}
Let $A$ be a  $m\times n$ ACI-matrix  over a  field $\mathbb{F}$ with $|\mathbb{F}|\geq \max\{m,n+1\}$.  Then $A$ has constant rank $\rho$  if and only if 
\begin{equation}\label{MatrixWithFixRank21}
A\sim  \small{\left[\begin{array}{ccc}
B & * & *  \\ 
0 & 0  & * \\ 
0  & 0 & C
\end{array}\right]}
\end{equation}
for some ACI-matrices   $B$ and $C$ which are square upper triangular    with nonzero constant diagonal entries   and whose orders sum to $\rho$. 
\end{theorem}

We remark that in~(\ref{MatrixWithFixRank21}) some block rows or/and block columns may be void.
Now in the next theorem we will  rewrite Theorem~\ref{ResultadosAnteriores0} making these degenerate cases more explicit by dividing the result in  different  cases  depending on the relation of $m$ and  $n$ with $\rho$. We will use square upper triangular ACI-matrices  with all its diagonal entries equal to 1  instead of square upper triangular ACI-matrices with nonzero constant diagonal entries. It is clear that this change can be done.

\begin{theoremWOnumbering2.4}
Let $A$ be a  $m\times n$ ACI-matrix of constant rank $\rho$ with $1 \leq \rho \leq \min \{m, n\}$ over a field  $\mathbb{F}$ with $|\mathbb{F}|\geq \max\{m,n+1\}$. Depending on  $m$, $n$ and $\rho$ we have the following possibilities:%\footnote{In the original statement of Theorem~\ref{ResultadosAnteriores} in~\cite{MR2775784} all  submatrices $\begin{smat} 1 & &* \\ &\ddots & \\ 0 & & 1 \end{smat}$ were written  as  $\begin{smat} a_1 & &* \\ &\ddots & \\ 0 & & a_n \end{smat}$ where  $a_1, \ldots, a_n$ are nonzero constants.   Clearly both statements  are equivalent.}:

\begin{enumerate}[(i)]
\item $\rho=m=n$ if and only if  
$
A \sim \begin{smat}  1 & & * \\   & \ddots &  \\ 0 & & 1 \end{smat}.
$

\item $\rho=m<n$ if and only if $ A \sim \scriptsize{ \begin{bmatrix}[ccc|cc] \nonumber
 1 &  & * &  \\ 
  & \ddots & & *  \\ 
 0 & & 1 &   
\end{bmatrix}}$.
\item $\rho=n<m$ if and only if  $ A \sim \scriptsize{\begin{bmatrix} 
  & * &  \\ \hline
 1 &  & *\\ 
  & \ddots &  \\
 0 &  & 1 \\ 
\end{bmatrix}}$.

\item $1\leq \rho<\min\{m,n\}$ if and only if for some positive integers $r$ and $s$ with $r+s=m+n-\rho$   
\begin{equation} \label{HZtype}
A \sim 
{\scriptsize \left[\begin{array}{cccc|ccc}
1 &  & \multicolumn{1}{c|}{*} & \multicolumn{1}{c|}{\multirow{3}{*}{\Large $\ * \ $}} &  \multicolumn{3}{c}{\multirow{3}{*}{\Large $*$}} \\ 
  & \ddots & \multicolumn{1}{c|}{} &  \\ 
 0 & & \multicolumn{1}{c|}{1} &  &  & & \\ \hline
 \multicolumn{4}{c|}{\multirow{4}{*}{\Large$0_{r\times s}$}} & \multicolumn{3}{c}{\multirow{1}{*}{\Large $*$}}\\ 
&&&&&& \\
 \cline{5-7}
 & & & & 1 &  & *\\ 
 & & & &  & \ddots &  \\
 & & & & 0 &  & 1 \\ 
\end{array}\right]}
\end{equation}
 where the upper blocks  do not appear if  $r=m$ and   the right blocks do not appear if  $s=n$. 
 \end{enumerate}
\end{theoremWOnumbering2.4}

\begin{remarkWOnumbering}
Note that if $\mathbb{F}$ is an infinite field then $|\mathbb{F}|\geq \max\{m,n+1\}$ and items $(i)$ to $(iv)$ are  satisfied.  In~\cite{BoCa1} we proved that  item $(i)$   is  true for any  field $\mathbb{F}$ without the restriction $|\mathbb{F}|\geq \max\{m,n+1\}$.   In~\cite[Lemma 2.1 and Lemma 3.1]{BoCa2}  we proved  the existence of   minimal full rank ACI-matrices and of maximal full rank ACI-matrices over  all finite fields  (see Example~\ref{mfrMFR}), and we showed that  if  $A$ is minimal  full rank   and  $B$ is   maximal  full rank  then  
$$A\not\sim {\scriptsize \begin{bmatrix}[ccc|cc] \nonumber
 1 &  & * & \\ 
  & \ddots & & *  \\ 
 0 & & 1 &   
\end{bmatrix}} \qquad \text{ and } \qquad  B \not \sim {\scriptsize \begin{bmatrix} 
  & * &  \\ \hline
 1 &  & *\\ 
  & \ddots &  \\
 0 &  & 1 \\ 
\end{bmatrix}}.$$
Therefore the results of items $(ii)$ and $(iii)$  of Theorem~\ref{ResultadosAnteriores0} can not be extended  to finite fields $\mathbb{F}$ disregarding completely the restriction on $|\mathbb{F}|$.  Moreover, in~\cite[Theorem 4.1 and Theorem 4.2]{BoCa2}   we  characterized the full rank ACI-matrices over arbitrary fields  with the help of the minimal and the maximal full rank ACI-matrices as can be seen in items  $(ii)$ and $(iii)$ of Theorem~\ref{ResultadosAnteriores3} below.  
\end{remarkWOnumbering}

%Independently from our work, McTigue and Quinlan (see~\cite{McTigueQuinlan3} and~\cite[Corollary 6.1]{McTigueTesis}) found a corollary of Theorem~\ref{ResultadosAnteriores}: if a partial matrix with constant rank $\rho$ is equivalent to an ACI-matrix of the types found in $(i), (ii), (iii)$ and $(iv)$, then it has a $\rho \times \rho$ submatrix of constant rank $\rho$. But, on the other hand, they also showed that for any prime power $q$ and positive integer $\rho<q$ there exist partial matrices over $\mathbb{F}_q$ with constant rank $\rho$ and no $\rho \times \rho$ partial submatrix of constant rank $\rho$. This means that $(ii), (iii)$ and $(iv)$ can not be extended to arbitrary fields. We will go back to their results later on, since they will be the motivation for Section~\ref{section3}.

It is worthy to mention the different approach taken by  McTigue and Quinlan (see~\cite{McTigueQuinlan3} and~\cite[Corollary 6.1]{McTigueTesis}). They proved that if a partial matrix $P$ with constant rank $\rho$ is equivalent to an ACI-matrix of the types found in $(i), (ii), (iii)$ and $(iv)$ of Theorem~\ref{ResultadosAnteriores0}, then $P$ has a $\rho \times \rho$ submatrix of constant rank $\rho$. On the other hand,  for any prime power $q$  they constructed   a  $(q+1)\times 2q$ partial matrix $A_q$ over $\mathbb{F}_q$ that has  constant rank $q+1$ and has no $(q+1)\times (2q-1)$ submatrix of constant rank $q+1$ (for instance $A_2= \begin{smat}
 1 & 1 & x_2 & 0 \\
 1 & 0 & 0 & x_3 \\ 
 1 & x_1 & 1 & 1 
 \end{smat}$). So $A_q$ has no $(q+1)\times (q+1)$ submatrix of constant rank $q+1$ and Theorem~\ref{ResultadosAnteriores0} can not be extended to $\mathbb{F}_q$ disregarding completely the restriction on $|\mathbb{F}_q|$. It is important to realize that  the mentioned conditions for $A_q$ imply that $A_q$ is  minimal full rank and so, interestingly, both approaches have led us to the same type of matrices.

%\begin{remark}
%In~\cite[Lemmas 2.1]{BoCa2} we proved that the maximal full rank ACI-matrices could not be characterized by Theorem~\ref{ResultadosAnteriores}~$(iii)$. Moreover, in~\cite[Theorem 4.1]{BoCa2} we presented a complete characterization of full rank ACI-matrices with more rows than columns over arbitrary fields.
%%$m\times n$ ACI-matrices of constant rank $n$ over arbitrary fields
%%Based on this fact we presented in~\cite[Theorem 4.1]{BoCa2} a sufficient and necessary condition for $m\times n$ ACI-matrices of constant rank $n$ over arbitrary fields. 
%
%In~\cite[Lemma 2.1]{BoCa2} we proved that the minimal full rank ACI-matrices could not be characterized by Theorem~\ref{ResultadosAnteriores}~$(ii)$. Based on this fact we presented in~\cite[Theorem 4.2]{BoCa2} a sufficient and necessary condition for $m\times n$ ACI-matrices of constant rank $m$ over arbitrary fields. 
%\end{remark}

One of our main objectives in this work is to complete the characterization of the  $m\times n$ ACI-matrices of constant rank $\rho$ over arbitrary fields.  This is done  in our next result that includes the case when $\rho<\min\{m,n\}$

\begin{theorem}\label{ResultadosAnteriores3}
Let $A$ be a  $m\times n$ ACI-matrix  over an arbitrary  field $\mathbb{F}$.  Then $A$ has constant rank $\rho$   if and only if 
\begin{equation}\label{MatrixWithFixRank7}
A\sim  \small{\left[\begin{array}{ccc}
B & * & *  \\ 
0 & 0  & * \\ 
0  & 0 & C
\end{array}\right]}
\end{equation}
for some ACI-matrices $B$ and $C$ such that  $B$  is   square upper triangular   with nonzero constant diagonal entries  or  is minimal full rank;  $C$  is  square upper triangular  with nonzero constant diagonal entries  or is maximal full rank; and  the number of rows of $B$ plus the number of columns of $C$ is $\rho$.  \end{theorem}

%\begin{theorem}\label{ResultadosAnteriores3}
%Let $A$ be a  $m\times n$ ACI-matrix of constant rank $\rho$ with $0 \leq \rho \leq \min \{m, n\}$ over an arbitrary  field $\mathbb{F}$.  Then there exist non-negative integers $m_1,m_2,n_2,n_2$ with $0\leq m_1+m_2\leq m$, $0\leq n_1+n_2 \leq n$ and  $m_1+n_2=\rho$ such that  
%\begin{equation*}\label{MatrixWithFixRank2}
%A\sim  \small{\left[\begin{array}{cc|c}
%B & * & *  \\ \hline
%0 & 0  & * \\ 
%0  & 0 & C
%\end{array}\right]}
%\end{equation*}
% where $B$ is $m_1\times n_1$ of constant rank $m_1$ and  $B=\begin{smat}  1 & & * \\   & \ddots &  \\ 0 & & 1 \end{smat}$ or $B$ is  minimal full rank, and $C$ is $m_2\times n_2$ of constant rank $n_2$ and   $C=\begin{smat}  1 & & * \\   & \ddots &  \\ 0 & & 1 \end{smat}$ or $C$ is   maximal full rank.
%\end{theorem}

We again remark that in~(\ref{MatrixWithFixRank7}) some block rows or/and block columns may be void.
 In the next theorem we will  rewrite Theorem~\ref{ResultadosAnteriores3} making these degenerate cases more explicit by dividing the result in  different  cases  depending on the relation of $m$ and  $n$ with $\rho$. 
%Once again the sizes of the ACI-matrices $B$ and $C$ are understood to be the adequate ones. In the next  theorem we  will rewrite Theorem~\ref{ResultadosAnteriores3} making explicit these sizes by dividing  the result in    different  cases depending on the relation of $m$ and $n$ with $\rho$. 
This will facilitate the proof of the result. And we will use square upper triangular ACI-matrices  with all its diagonal entries equal to 1  instead of square upper triangular ACI-matrices with nonzero constant diagonal entries.

%Now we will complete the characterization of  ACI-matrices of constant rank over arbitrary fields.

\begin{theoremWOnumbering2.5}
Let $A$ be a  $m\times n$ ACI-matrix of constant rank $\rho$ with $1 \leq \rho \leq \min \{m, n\}$ over an arbitrary  field $\mathbb{F}$. Depending on  $m$, $n$ and $\rho$ we have the following possibilities:
\begin{enumerate}[(i)]
\item $\rho=m=n$ if and only if 
$
A \sim \begin{smat}  1 & & * \\   & \ddots &  \\ 0 & & 1 \end{smat}.
$

\item  $\rho=m<n$ if and only if $ A \sim \small{\begin{bmatrix}  B & * \, \end{bmatrix}}$ where either  $B= { \begin{smat} 
 1 &  & *  \\ 
  & \ddots &   \\ 
 0 & & 1  
\end{smat}}$ or     $B$ is $m\times n'$ minimal full rank   with $m <  n'\leq n$.

\item  $\rho=n<m$ if and only if  $ A \sim {\begin{smat}  * \\ \\   C   \end{smat}}$ where either $C= {\begin{smat} 
 1 &  & *\\ 
  & \ddots &  \\
 0 &  & 1 \\ 
\end{smat}}$ or $C$ is $m'\times n$ maximal full rank   with $n< m' \leq m$.

\item $ \rho<\min\{m,n\}$ if and only if one of the following possibilities is satisfied:

\begin{enumerate}[$(a)$]
\item there exist positive integers $r<m$ and $s<n$ with $\rho=(m-r)+(n-s)$ such that  
\begin{equation*}\label{MatrixWithFixRank2}
A\sim  \small{\left[\begin{array}{cc|c}
B & * & *  \\ \hline
\multicolumn{2}{c|}{\multirow{2}{*}{$0_{r\times s}$}} & * \\ 
 & & C
\end{array}\right]}
\end{equation*}
 where $\small{\begin{bmatrix}  B & * \, \end{bmatrix}}$ has less rows than columns with either   $B=\begin{smat}  1 & & * \\   & \ddots &  \\ 0 & & 1 \end{smat}$ or $B$ is  minimal full rank, and ${\begin{smat}  *  \\  \\  C   \end{smat}}$ has more rows than columns with either $C=\begin{smat}  1 & & * \\   & \ddots &  \\ 0 & & 1 \end{smat}$ or $C$ is   maximal full rank.
 
\item there exists a positive integer $r<m$  with $\rho=m-r$ such that $A\sim  \small{\left[\begin{array}{cc}
B & *   \\ \hline
\multicolumn{2}{c|}{\multirow{1}{*}{$0_{r\times n}$}}  
\end{array}\right]}$ where $\small{\begin{bmatrix}  B & * \, \end{bmatrix}}$ has less rows than columns with either   $B=\begin{smat}  1 & & * \\   & \ddots &  \\ 0 & & 1 \end{smat}$ or $B$ is  minimal full rank.

\item there exists a positive integer $s<n$  with $\rho=n-s$ such that $A\sim  \small{\left[\begin{array}{cc|c}
\multicolumn{2}{c|}{\multirow{2}{*}{$0_{m\times s}$}} & * \\ 
 & & C \end{array}\right]}$  where
${\begin{smat}  *  \\  \\  C   \end{smat}}$ has more rows than columns with either $C=\begin{smat}  1 & & * \\   & \ddots &  \\ 0 & & 1 \end{smat}$ or $C$ is   maximal full rank.
\end{enumerate}

 \end{enumerate}
\end{theoremWOnumbering2.5}

\begin{proof}
Item $(i)$ was proved in~\cite[Theorem 3.1]{BoCa1} and items $(ii)$ and $(iii)$ were proved  in~\cite[Theorems 4.1 and 4.2]{BoCa2}. Let us prove item $(iv)$:
\begin{description}
\item[$\Rightarrow)$] By Corollary~\ref{SubmatrizCeros3},  for some positive integers $r$ and $s$ with $\rho=(m-r)+(n-s)$ there exist a nonsingular constant $T$ of order $m$ and a permutation matrix $Q$ of order $n$ such that 
$$TAQ =  \small{\left[\begin{array}{c|c}A_{11} & A_{12} \\ \hline 0_{r\times s} & A_{22} \end{array}\right]}$$ 
where   the upper blocks  do not appear if  $r=m$ and   the right blocks do not appear if  $s=n$.

Corollary~\ref{SubmatrizCeros3} asserts that if $r<m$ then   $A_{11}$  is  $(m-r)\times s$ with $\rank(A_{11})=\{m-r\}$. As   $m-r=\rho-n+s<s$ then by item $(ii)$ of this theorem, there exist a nonsingular constant  $T_1$ of order $m-r$ and a permutation matrix $Q_1$ of order $s$ such that $ T_1A_{11}Q_1 = \small{\begin{bmatrix}  B & * \, \end{bmatrix}}$ where  either $B=\begin{smat}  1 & & * \\   & \ddots &  \\ 0 & & 1 \end{smat}$ or $B$ is a $(m-r) \times s'$  minimal full rank with $m-r <  s'\leq s$.

 Corollary~\ref{SubmatrizCeros3} also asserts that if $s<n$ then     $A_{22}$ is $r\times (n-s)$  with $\rank(A_{22})=\{n-s\}$. As   $n-s=\rho-m+r<r$ then by item  $(iii)$ of this theorem, there exist a nonsingular constant  $T_2$ of order $r$ and a permutation matrix $Q_2$ of order $n-s$ such that $ T_2A_{22}Q_2 =  {\begin{smat}  *  \\  \\  C   \end{smat}}$ where   either $C=\begin{smat}  1 & & * \\   & \ddots &  \\ 0 & & 1 \end{smat}$ or $C$ is a $r'\times (n-s)$ maximal full rank with $n-s< r' \leq r$. 

Observe that $r=m$ and $s=n$  is not  possible. So we consider  three  cases:
\begin{enumerate}[$(a)$]
\item if $r<m$ and $s<n$ then 
$$A\sim  TAQ \sim \small{\left[\begin{array}{c|c}T_{1} & 0 \\ \hline 0 & T_{2} \end{array}\right] TAQ 
\left[\begin{array}{c|c}Q_{1} & 0 \\ \hline 0 & Q_{2} \end{array}\right] =
\left[\begin{array}{c|c}T_1A_{11}Q_1 & * \\ \hline 0_{r\times s} & T_2A_{22}Q_2 \end{array}\right]=
\left[\begin{array}{cc|c}
B & * & *  \\ \hline
\multicolumn{2}{c|}{\multirow{2}{*}{$0_{r\times s}$}} & * \\ 
 & & C
\end{array}\right]}.$$ 

\item  if $r<m$ and $s=n$ then 
\[
A\sim  
\small{\left[\begin{array}{c|c}T_{1} & 0 \\ \hline 0 & I_r \end{array}\right]} (TAQ) Q_1 =
\small{\left[\begin{array}{c|c}T_{1} & 0 \\ \hline 0 & I_r \end{array}\right]} \left[\begin{array}{c} A_{11} \\ \hline 0_{r\times n} \end{array}\right] Q_1 =
\left[\begin{array}{c} T_1 A_{11} Q_1 \\ \hline 0_{r\times n} \end{array}\right]=
\left[\begin{array}{c} B \; \; \; \; *  \\ \hline 0_{r\times n} \end{array}\right].
\]

\item if $r=m$ and $s<n$ then
\[
A\sim  
T_2 (TAQ) \left[\begin{array}{c|c}I_s & 0 \\ \hline 0 & Q_{2} \end{array}\right] =
T_2 \left[\begin{array}{c|c} 0_{m\times s} & A_{22} \end{array}\right] \left[\begin{array}{c|c}I_s & 0 \\ \hline 0 & Q_{2} \end{array}\right] =
\left[\begin{array}{c|c} 0_{m\times s} & T_2A_{22}Q_2 \end{array}\right]=
\left[\begin{array}{cc|c}
\multicolumn{2}{c|}{\multirow{2}{*}{$0_{m\times s}$}} & * \\ 
 & & C
\end{array}\right].
\] 
\end{enumerate}

\item[$\Leftarrow)$] We will prove  for each one of the three  cases that $\rho<\min\{m,n\}$:
\begin{enumerate}[$(a)$]

\item By hypothesis 
$A\sim  \small{\left[\begin{array}{cc|c}
B & * & *  \\ \hline
\multicolumn{2}{c|}{\multirow{2}{*}{$0_{r\times s}$}} & * \\ 
 & & C
\end{array}\right]}$. 
As $\small{\begin{bmatrix}  B & * \, \end{bmatrix}}$
 has less rows than  columns then $m-r<s$ and  so
 $$
 \rho=(m-r)+(n-s)<s+(n-s)=n.
 $$
On the other hand, as ${\begin{smat}  *  \\  \\  C   \end{smat}}$
 has more rows than columns then  $r>n-s$ and  so 
 $$
 \rho=(m-r)+(n-s)<(m-r)+r=m.
 $$

\item By hypothesis $A\sim  \small{\left[\begin{array}{cc}
B & *   \\ \hline
\multicolumn{2}{c|}{\multirow{1}{*}{$0_{r\times n}$}}  
\end{array}\right]}$.
As $\small{\begin{bmatrix}  B & * \, \end{bmatrix}}$ has less rows than  columns then $m-r<n$, so $\rho=m-r  <\min\{m,n\}.$

\item By hypothesis
$A\sim  \small{\left[\begin{array}{cc|c}
\multicolumn{2}{c|}{\multirow{2}{*}{$0_{m\times s}$}} & * \\ 
 & & C \end{array}\right]}$.
 As  ${\begin{smat}  *  \\  \\  C   \end{smat}}$ has more rows than columns then       $n-s<m$, so 
 $\rho=n-s  <\min\{m,n\}.$
\end{enumerate}
\end{description}
\end{proof}

\begin{note} 
If $A$ is a constant matrix   then $A$ is equivalent to one of the following constant matrices: 
$$
\scriptsize{ \begin{bmatrix}[ccc] \nonumber
 1 &  & 0  \\ 
  & \ddots &    \\ 
 0 & & 1   
\end{bmatrix}}, \ 
 \scriptsize{ \begin{bmatrix}[ccc|cc] \nonumber
 1 &  & 0 &  \\ 
  & \ddots & & *  \\ 
 0 & & 1 &   
\end{bmatrix}},  \ 
 \scriptsize{\begin{bmatrix} 
 1 &  & 0\\ 
  & \ddots &  \\
 0 &  & 1 \\ \hline
   & 0 &  \\ 
\end{bmatrix}}, \  \scriptsize{\left[\begin{array}{ccc|c}
 1 &  & 0 &  \\ 
  & \ddots & & * \\
 0 &  & 1 &  \\ \hline 
 & 0 & & 0
\end{array}\right]}
$$
which are  in row reduced echelon form and correspond, respectively, to the patterns  given in  items $(i)$, $(ii)$, $(iv)(b)$ and $(iv)(a)$ of Theorem~\ref{ResultadosAnteriores3}. So, in a sense Theorems~\ref{ResultadosAnteriores0}  and~\ref{ResultadosAnteriores3} provide a generalization of the row reduced echelon form (extended by columns permutation) for ACI-matrices.
\end{note}

\subsection{An example}

Consider the  $7\times 7$ ACI-matrix  over $\mathbb{F}_2$ 
\begin{align*}
A=\left[
\begin{array}{ccccccc}
 x_1+y_1 & x_2+1 & 1 & x_4 & 0 & x_6+1 & x_7 \\
 x_1 & 1 & 1 & x_4 & 1 & x_6+1 & x_7 \\
 x_1 & x_2+1 & 0 & 0 & x_5 & 0 & 0 \\
 y_1+1 & y_2 & x_3 & y_4 & 0 & x_6 & 1 \\
 0 & x_2+y_2+1 & x_3 & y_4 & 0 & x_6 & 1 \\
 1 & x_2 & 1 & x_4 & 0 & x_6+1 & x_7 \\
 y_1 & x_2+y_2+1 & x_3 & y_4 & x_5 & x_6 & 1 \\
\end{array}
\right]
\end{align*}
It has $10$ variables and each variable can take 2 values. With a computer it is easy to calculate the rank of the $2^{10}$ different completions of $A$ to conclude  that $A$ has constant rank $\rho=5$. Our  intention is  to find an ACI-matrix equivalent to  $A$ which is expressed as equation~(\ref{MatrixWithFixRank7}) given in  Theorem~\ref{ResultadosAnteriores3}. 
%As $A$ has 10 variables and each variable admits 2 completions then $A$ admits $2^{10}$ completions. After calculating the rank of all completions we would  conclude that $A$ has constant rank 5. Nevertheless, we explain an alternative method that could reduce the calculations when the number of variables and the number of elements of the field are greater. 

Consider any variable of $A$, for instance $x_1$.  Permute  rows and columns of $A$ so that  $x_1$ is placed  in the (1,1)-position. In this case  no permutation of rows or columns is necessary. After that delete $x_1$ from the rest of entries of the first column by multiplying by the left by an adequate nonsingular ACI-matrix $T$. So variable $x_1$ only appears in the (1,1)-entry of the ACI-matrix $A_1=TA$ with $A_1\sim A$. Consider now any variable that is neither in the first row nor in the first column of $A_1$, for instance $x_2$. We can proceed as before so that we will obtain  $A_2\sim A_1$ such that $x_1$ only appears on the (1,1)-position of $A_2$ and  $x_2$ only appears on the  (2,2)-position of $A_2$. Repeat this process with any variable that is neither in the  first two rows nor in the  first two columns of $A_2$. And so on until no variable remains.  At the end of this procedure   we obtain 
\begin{align} \label{diagonalPivots}
A\sim A_5=\left[
\begin{array}{ccccccc}
\text{\circled{$x_1$}} +y_1+1 & 1 & \mathbf{0} & \mathbf{0} & 0 & \mathbf{0} & \mathbf{0} \\ \cline{2-7}
 y_1 & \multicolumn{1}{|c}{\text{\circled{$x_2$}}} & \mathbf{0} & \mathbf{0} & 1 & \mathbf{0} & \mathbf{0} \\ \cline{3-7}
 y_1+1 & \multicolumn{1}{|c}{y_2} & \multicolumn{1}{|c}{\text{\circled{$x_3$}}} & y_4 & 0 & x_6 & 1 \\ \cline{4-7}
 y_1+1 & \multicolumn{1}{|c}{0} & \multicolumn{1}{|c}{1} & \multicolumn{1}{|c}{\text{\circled{$x_4$}}} & 1 & x_6+1 & x_7 \\ \cline{5-7}
 1 & \multicolumn{1}{|c}{0} & \multicolumn{1}{|c}{\mathbf{0}} & \multicolumn{1}{|c}{\mathbf{0}} & \multicolumn{1}{|c}{\text{\circled{$x_5$}}+1} & \mathbf{0} & \mathbf{0} \\ \cline{6-7}
 1 & \multicolumn{1}{|c}{1} & \multicolumn{1}{|c}{\mathbf{0}} & \multicolumn{1}{|c}{\mathbf{0}} & \multicolumn{1}{|c}{1} & \multicolumn{1}{|c}{\mathbf{0}} & \mathbf{0} \\
 y_1 & \multicolumn{1}{|c}{1} & \multicolumn{1}{|c}{\mathbf{0}} & \multicolumn{1}{|c}{\mathbf{0}} & \multicolumn{1}{|c}{0} & \multicolumn{1}{|c}{\mathbf{0}} & \mathbf{0} \\
\end{array}
\right]
\end{align}
where we have circled the chosen variables $x_1, x_2, x_3, x_4, x_5$, which we will call \emph{pivots}.
%For $i=1,\ldots,5$ consider the ACI-matrix $A^{(i)}_5$ obtained by deleting the first $i$ rows and the first $i$ columns of $A_5$.  
%Observe that for $k=1,\ldots,5$ the pivot  in the (1,1)-position of $A^{(k-1)}_5$ does not appear in any other position of $A_5$ and so $\Mrank(A^{(k-1)}_5)\geq \Mrank(A^{(k)}_5+1)$. Therefore 
%\begin{equation}\label{Mranks}
%\Mrank(A)=\Mrank(A^{(0)})\geq\Mrank(A^{(1)})+1\geq\cdots\geq \Mrank(A^{(5)})+5=5
%\end{equation}
%By the inequalities given in~(\ref{Mranks}) we know that $\rho\geq 5$. 

By Corollary~\ref{SubmatrizCeros3} we know that there exists positive integers $r$ and $s$ with 
$$r+s=7+7-5=9$$ 
such that 
\begin{align*}%\label{EquivZeroBlock}
A_5\sim \left[\begin{array}{c|c}A_{11} & A_{12} \\ \hline 0_{r\times s} & A_{22} \end{array}\right].
\end{align*} 
%As  $A_5$ has no column of zeroes then $r<7$ and $A_{11}$ has size $(7-r)\times s$ with $\rank(A_{11})=\{7-r\}$. As no row of $A_5$ is a linear combination of the remaining ones then $s<7$   and $A_{22}$ has size $r\times (7-s)$ wih $\rank(A_{22})=\{7-s\}$. 

It is clear that this equivalence can be realized through a permutation of rows and columns\footnote{In~(\ref{diagonalPivots}) we have marked the zeros in bold face so that, when reordered, give the zero block $0_{r\times s}$ of~(\ref{TheBlockOfZeros})}. This is not a coincidence. Although it is not necessary, we will prove it \emph{formally} %that the equivalence in~(\ref{EquivZeroBlock}) can be realized through a permutation of rows and columns 
so that it is  possible to extrapolate the arguments whenever  the number of pivots and the constant rank are equal. 

%We want the reader to keep in mind that  the reasoning below is particular,

Let $x_{i_1},\ldots,x_{i_h}$ be the pivots that appear in $A_{11}$.  Proceeding with $A_{11}$ as we did with $A$ we obtain that $ A_{11} \sim T_1A_{11}Q_1=A'_{11}$ with $x_{i_v}$ only appearing in the  $(v,v)$-position of $A'_{11}$. In the same way, let $x_{j_1},\ldots,x_{j_k}$ the pivots that appear in $A_{22}$.  Again, proceeding with $A_{22}$ as we did with $A$ we obtain $A_{22}\sim T_2A_{22}Q_2=A'_{22}$ with $x_{j_w}$ only appearing in the  $(w,w)$-position of $A'_{22}$. So 
$$
\left[\begin{array}{c|c}T_{1} & 0 \\ \hline 0 & T_{2} \end{array}\right]
\left[\begin{array}{c|c}A_{11} & A_{12} \\ \hline 0_{r\times s} & A_{22} \end{array}\right]
\left[\begin{array}{c|c}Q_{1} & 0 \\ \hline 0& Q_{2} \end{array}\right]=
%\left[\begin{array}{c|c}T_1A_{11}Q_1 & T_1A_{12}Q_2 \\ \hline 0_{r\times s} & T_2A_{22}Q_2 \end{array}\right]=
\left[\begin{array}{c|c}A'_{11} & T_1A_{12}Q_2 \\ \hline 0_{r\times s} & A'_{22} \end{array}\right].
$$
Finally, if a row of $T_1A_{12}Q_2$ contains  the pivot  $x_{j_t}$ for some $t=1,\ldots,k$ then we can delete it by adding a multiple of the $t$-th row of $A'_{22}$. Let $A'_{12}$ be the ACI-matrix obtained after we have deleted the pivots $x_{j_1},\ldots,x_{j_k}$ of $T_1A_{12}Q_2$. Then 
$$
\left[\begin{array}{c|c}A'_{11} & T_1A_{12}Q_2 \\ \hline 0_{r\times s} & A'_{22} \end{array}\right]\sim
\left[\begin{array}{c|c}A'_{11} & A'_{12} \\ \hline 0_{r\times s} & A'_{22} \end{array}\right]=A'.
$$
In $A'_{12}$ there will remain $l=5-h-k$ pivots that will appear in $l$ rows of $A'_{12}$ that are  different from its first $h$ rows and in  $l$ columns  of $A'_{12}$ that are  different from its first $k$ columns.  
As $A'_{11}$ has $7-r$ rows  and $A'_{22}$ has $7-s$  columns then 
$$5=(7-r)+(7-s)\geq (h+l)+(k+l)=5+l$$
and so  $l=0$,  $h=7-r$ and $k=7-s$.  The 5 pivots in $A'$ are in the $(1,1),\ldots(7-r,7-r)$ positions of $A'_{11}$ and in the $(1,1),\ldots(7-s,7-s)$ positions of   $A'_{22}$. 

Note that the pivots are in the first 5 rows of $A'$. So each of the first 5 rows of $A_5$  only participates in the row of $A'$ in which the same pivot appears. In other words, there exists a permutation $P$ of order 5 such that 
$$A'=\left[\begin{array}{cc}P & T_{12} \\  0_{2\times 5} & T_{22} \end{array}\right]
 A_5 Q \qquad \text{with} \ \det(T_{22})\neq 0
$$
%$$A'=TA_5Q=
%\left[\begin{array}{ccccc|cc} 
%& & & & & t_{16} & t_{17} \\  
%& & & & & t_{26} & t_{27} \\  
%& & P & & & t_{36} & t_{37} \\  
%& & & & & t_{46} & t_{47} \\  
%& & & & & t_{56} & t_{57} \\  \hline
%0 & 0 & 0 & 0 & 0 & t_{66} & t_{67} \\  
%0 & 0 & 0 & 0 & 0 & t_{76} & t_{77} \\  
% \end{array}\right] A_5 Q \qquad \text{with} \ \det \left[\begin{array}{cc}t_{66} & t_{67} \\  t_{76} & t_{77} \end{array}\right]\neq 0
%$$
The position of the pivots of $A'$ does not depend of $T_{12}$ and $T_{22}$. On the other hand, $r+s=9$ implies that $r\geq 2$ and so the $0_{r\times s}$ block of $A'$ does not depend of $T_{12}$ and $T_{22}$.  So without loss of generality we can assume the simplest situation:  $T_{12}=0_{5\times 2}$ and $T_{22}=I_2$. Therefore  $A'$ can be obtained by permuting the columns of one  ACI-matrix which in turn is obtained by permuting the first 5 rows of  $A_5$. This means that the zeros of the zero block we are looking for should already appear in~(\ref{diagonalPivots}). It should be straightforward to find such a block of zeros.

%We start an exhaustive search to calculate $A'$. Firstly,  suppose that $r=6$ and  $s=3$. For $1\leq i\leq 5$ consider  a permutation of rows and columns of $A_5$ such that   the pivot $x_i$ goes to the $(1,1)$-position and the rest of the pivots  goes  in  order to the $(2,4)$, $(3,5)$, $(4,6)$ and $(5,7)$ positions. It is easy to see that none case the left and down $6\times 3$ block is $0_{6\times 3}$. Secondly, suppose that $r=5$ and  $s=4$. For $1\leq i < j \leq 5$ consider  a permutation of rows and columns of $A_5$ such that   the pivots $x_i$ and $x_j$ goes to the $(1,1)$ and $(2,2)$ positions and the rest of the pivots  goes in  order to the $(3,5)$, $(4,6)$ and $(5,7)$ positions. For $x_3$ and $x_4$ we stop since we obtain 
%After a suitable rearranging of rows and columns of $A_5$ we can allocate a $5\times 4$ block of zeros on the left down corner:
\begin{align} \label{TheBlockOfZeros}
A'=\left[\begin{array}{c|c}A'_{11} & A'_{12} \\ \hline 0_{5\times 4} & A'_{22} \end{array}\right]=\left[
\begin{array}{ccccccc}
 \text{\circled{$x_3$}} & y_4 & x_6 & 1 & y_1+1 & y_2 & 0 \\  
 1 & \text{\circled{$x_4$}} & x_6+1 & x_7 & y_1+1 & 0 & 1 \\ \cline{1-4}
 \mathbf{0} & \mathbf{0} & \mathbf{0} & \multicolumn{1}{c|}{\mathbf{0}} & \text{\circled{$x_1$}} +y_1+1 & 1 & 0 \\
 \mathbf{0} & \mathbf{0} & \mathbf{0} & \multicolumn{1}{c|}{\mathbf{0}} & y_1 & \text{\circled{$x_2$}} & 1 \\
 \mathbf{0} & \mathbf{0} & \mathbf{0} & \multicolumn{1}{c|}{\mathbf{0}} & 1 & 0 & \text{\circled{$x_5$}}+1 \\
 \mathbf{0} & \mathbf{0} & \mathbf{0} & \multicolumn{1}{c|}{\mathbf{0}} & 1 & 1 & 1 \\
 \mathbf{0} & \mathbf{0} & \mathbf{0} & \multicolumn{1}{c|}{\mathbf{0}} & y_1 & 1 & 0 \\
\end{array}
\right]
\end{align}

\emph{This procedure for finding the block of zeros in some ACI-matrix equivalent to $A$, can be applied to other examples whenever we are given a $\rho$ constant rank ACI-matrix for which we can find  $\rho$ pivots}. In the present example $A$ has constant rank $\rho=5$ and we have found 5 pivots. 

Now we continue our search of an equivalent ACI-matrix of $A$ that is of type~(\ref{MatrixWithFixRank7}) of  Theorem~\ref{ResultadosAnteriores3}:
\begin{itemize}
\item[$\bullet$] By checking all completions of $A'_{11}$  we know that $A'_{11}$   has constant rank 2.
%From Corollary~\ref{SubmatrizCeros3} it follows that $A'$, and therefore $A$, has constant rank 5. 
Moreover, as $A'_{11}$  has no constant column then $A'_{11}\not\sim {\small \left[\begin{array}{cc|cc}1 & * & * & *  \\  0 & 1 & * & * \end{array}\right]}$. So, by Theorem~\ref{ResultadosAnteriores3}~$(ii)$, $A'_{11}\sim {\small \left[\begin{array}{cc} B & *   \end{array}\right]}$ where $B$ is minimal full rank. Note that if $x_4=y_4=0$ then the second column of  $A'_{11}$ is null, so this column can not be part of a minimal full rank ACI-matrix. Moreover, in Example~\ref{mfrMFR}~$(i)$ we saw that  ${\small \left[\begin{array}{ccc} x_3  & x_6 & 1  \\  1  & x_6+1 & x_7 \end{array}\right]}$ is minimal full rank. Therefore 
$$A'_{11}\sim {\small \left[\begin{array}{cc} B & *   \end{array}\right]} = { \left[\begin{array}{ccc|c} x_3  & x_6 & 1 & y_4 \\  1  & x_6+1 & x_7 & x_4 \end{array}\right]}$$

\item[$\bullet$] On the other hand,  by checking all completions of $A'_{22}$  we known that $A'_{22}$ has constant rank 3. 
 Moreover,   $A'_{22}\not\sim {\tiny \left[
\begin{array}{ccc} * & * &  * \\ * & *  & * \\ \hline 1 & * & * \\  0 & 1 & * \\ 0 & 0 & 1 \end{array}
\right]}$ since no linear combination of the rows of $A'_{22}$ is equal to either   $[\ 1 \ 0 \ 0 \ ]$ or $[\ 0 \ 1 \ 0 \ ]$ or $[\ 0 \ 0 \ 1 \ ]$. So, by Theorem~\ref{ResultadosAnteriores3}~$(iii)$, $A'_{22}\sim {\begin{smat}  * \\ \\   C   \end{smat}}$ where $C$ is maximal full rank. We want to know if $A'_{22}$ is maximal full rank, so we proceed to check  if some augmented ACI-matrix $[ \ A'_{22} \ | \  v \ ]$ with $v\in \mathbb{F}_2^5$ has constant rank. We discover  that in fact  
\begin{equation}\label{ParteMaximal}
\rank \left[\begin{array} {ccc|c}
x_1+y_1+1 & 1 & 0 & \ 1 \ \\  
y_1 & x_2 & 1 & 0 \\ 
1 & 0 & x_5+1 & 0 \\ 
1 & 1 & 1 & 0 \\ 
y_1 & 1 & 0 & 0
\end{array}\right]=\{4\}
\end{equation}
and therefore $A'_{22}$ is not maximal full rank. From Equation~(\ref{ParteMaximal}) it follows that  deleting the first row of $A'_{22}$ we obtain an ACI-matrix, that we will denote $C$, such that 
$$\rank(C)=\rank \left[\begin{array} {cccc}
y_1 & x_2 & 1  \\ 
1 & 0 & x_5+1  \\ 
1 & 1 & 1  \\ 
y_1 & 1 & 0 
\end{array}\right]=\{3\} $$
Again we proceed to check  if some  augmented ACI-matrix $[ \ C \ | \ w \ ]$ with $w\in \mathbb{F}_2^4$ has constant rank 4. As this is not the case then $C$ is maximal full rank. 
\end{itemize}

Therefore we conclude that 
\begin{align*}
A\sim  \small{\left[\begin{array}{c|c|c}
B & * & *  \\ \hline
0 & 0  & * \\ \hline
0  & 0 & C
\end{array}\right]} \sim 
\left[\begin{array}{ccc|c|ccc}
 \text{\circled{$x_3$}}  & x_6 & 1 & y_4 & y_1+1 & y_2 & 0 \\  
 1  & x_6+1 & x_7 & \text{\circled{$x_4$}} & y_1+1 & 0 & 1 \\ \hline
 \mathbf{0} & \mathbf{0} & \mathbf{0} & \mathbf{0}& \text{\circled{$x_1$}} +y_1+1 & 1 & 0 \\ \hline
 \mathbf{0} & \mathbf{0} & \mathbf{0} & \mathbf{0} & y_1 & \text{\circled{$x_2$}} & 1 \\
 \mathbf{0} & \mathbf{0} & \mathbf{0} & \mathbf{0} & 1 & 0 & \text{\circled{$x_5$}}+1 \\
 \mathbf{0} & \mathbf{0} & \mathbf{0} & \mathbf{0} & 1 & 1 & 1 \\
 \mathbf{0} & \mathbf{0} & \mathbf{0} & \mathbf{0} & y_1 & 1 & 0 \\
\end{array}
\right]
\end{align*}
where $B$ is minimal full rank and $C$ is maximal full rank, a possibility that Theorem~\ref{ResultadosAnteriores3} considers.

\begin{remark} \label{remarkSeccion2}
%In the previous example, in order to check that $A$ has constant rank 5, we can avoid to calculate the rank of $2^{10}=1024$ completions as we mentioned in the beginning. 
In order to check that $A$ has constant rank, we can avoid checking the rank of all its completions ($2^{10}=1024$ as we pointed out in the beginning of the subsection) and instead check the rank of two much smaller ACI-matrices with less variables.
We proceed  as follows. We take any completion, for instance the one obtained  by assigning 0 to all variables, which will give us a rank of 5. So if $\rank(A)=\{\rho\}$ then $\rho=5$. We assume  that $A$ has constant rank 5 and we proceed in the same way that we did in this subsection. At some point we prove that $A'_{11}$ has constant rank 2 and $A'_{22}$ has constant rank 3. We conclude, from Lemma~\ref{SC+FRC=constant-rank}, that $\rank(A)=\rank(A')=\{5\}$. Note that for  calculating the ranks of $A'_{11}$  and of $A'_{22}$ we calculate the rank of  all their  completions. But the number of completions of $A'_{11}$ is $2^5=32$ and the number of completion of $A'_{22}$ is $2^4=16$. These numbers are much smaller than 1024. Moreover,  the size of  $A'_{11}$ and  $A'_{22}$ are  smaller than the size of $A$.
\end{remark}

\section{The concept of reducibility for  ACI-matrices}

If A is a $m\times n$ constant matrix of rank $\rho$ then it is well known that we can delete $m-\rho$ rows and $n-\rho$ columns in such a way that the $\rho \times \rho$ submatrix of $A$ that we obtain has rank $\rho$. 
%And if $B$ is   equivalent to $A$ then $B$ is a $m\times n$ constant matrix of rank $\rho$ that shares this property. 
We would like to know if ACI-matrices share this property. First we will consider partial matrices. 
%Partial matrices are a subclass of ACI-matrices, so the definition of rank for partial matrices is inherited from the one for ACI-matrices.  
We might naively expect that any partial matrix  with constant rank $\rho$ must also have a $\rho\times \rho$ submatrix of constant rank $\rho$. McTigue and Quinlan studied the rank of partial matrices in~\cite{McTigueQuinlan,McTigueQuinlan2,McTigueQuinlan3} and   proved  that this is not the case. 
\begin{theorem}(\cite{McTigueQuinlan3})\label{PartialSubmatrices}
Every partial matrix $A$ of constant rank $\rho$ over a field $\mathbb{F}$ possesses an $\rho\times \rho$  submatrix of constant rank $\rho$ if and only if $|\mathbb{F}|\geq \rho$. 
\end{theorem}

%It follows from Theorem~\ref{PartialSubmatrices} that not all partial matrices of constant rank $\rho$ have  a submatrix of size  $\rho\times \rho$ of constant rank $\rho$. Indeed, 
They showed that if $|\mathbb{F}|< \rho$ then there exist examples of partial matrices of size $m\times n$ with $\max\{m,n\}\geq \rho+|\mathbb{F}|-1$ that does not contain a $\rho\times \rho$ submatrix with rank $\rho$. For $\rho=3$ and $\mathbb{F}_2$  they provided the following   $4\times 3$ partial matrix 
\begin{align} \label{thePartialMFR}
{\small P= \begin{bmatrix}[cccccc]
 1 & 1 & 1 \\ 
 0 & 1 & x_3 \\
x_1 & 0 & 1 \\
 0   & x_2 & 1 \\
 \end{bmatrix}}
\end{align}
that has constant rank 3 and has no $3\times 3$ submatrix of constant rank 3. %In other words, we can not delete a row of $P$ in such a way that the rank remains equal to 3.

Now we will state some definitions motivated by the previous remarks. Since partial matrices are a subclass of ACI-matrices, these definitions will be stated in the more general framework.

\begin{definition} \label{irreducible}
Let $A$ be a $m\times n$ ACI-matrix of constant rank $\rho$ over a field $\mathbb{F}$. We  say that:
\begin{enumerate}
\item  $A$  is {\bf row reducible} if it contains some row $R$ such that the $(m-1)\times n$ ACI-matrix obtained by deleting $R$ from $A$  has constant rank $\rho$. And $A$ is {\bf row irreducible} otherwise.

\item    $A$   is {\bf column reducible} if it contains some column $C$ such that the $m\times (n-1)$ ACI-matrix obtained by deleting $C$ from $A$  has constant rank $\rho$. And  $A$ is {\bf column irreducible} otherwise. 

\item $A$  is {\bf  reducible} if it is is row reducible and/or column reducible.  And  $A$ is {\bf  irreducible} otherwise. 
\end{enumerate}
\end{definition}

With this terminology the partial matrix $P$ given in~(\ref{thePartialMFR}) is irreducible. If we consider $P$ as an ACI-matrix we might expect to find, in the equivalence class of $P$, some reducible ACI-matrix. That is, some ACI-matrix with a $3\times 3$ ACI-submatrix of constant rank 3. But this is not the case. 

Nevertheless, there are  irreducible partial matrices such that its equivalence class contains reducible ACI-matrices. For instance, over $\mathbb{F}_2$ consider 
\begin{align} \label{partialIrred->ACIred}
\scriptsize
E=\begin{bmatrix}
 1 & 0 & 0 & 0 & 1 \\
 0 & 1 & 1 & 1 & 0 \\
 x_1 & 1 & 0 & 0 & 0 \\
 1 & x_2 & 0 & 0 & 0 \\
 1 & 1 & x_3 & 1 & 0 \\
 1 & 0 & 0 & x_4 & 0 \\
 1 & 0 & 0 & 1 & x_5 \\
\end{bmatrix} 
\xrightarrow[\sim]{ \stackrel{\mbox{row}_4 \rightarrow \mbox{row}_4 + \mbox{row}_1}{\mbox{row}_7 \rightarrow \mbox{row}_7 + \mbox{row}_1} }
F=\begin{bmatrix}
 1 & 0 & 0 & 0 & 1 \\ \hline
 0 & 1 & 1 & 1 & 0 \\
 x_1 & 1 & 0 & 0 & 0 \\
 0 & x_2 & 0 & 0 & 1 \\
 1 & 1 & x_3 & 1 & 0 \\
 1 & 0 & 0 & x_4 & 0 \\
 0 & 0 & 0 & 1 & x_5+1 \\
\end{bmatrix}.
\end{align}
It can be checked that the partial matrix $E$ is irreducible and has constant rank 5, that   $F$ is equivalent to $E$ and so has constant rank 5, and that  $F$ is row reducible: if we delete its first row we obtain an ACI-matrix of constant rank 5.\footnote{If we do not impose to the  irreducible matrix  to be  partial  then there are much simpler examples than~(\ref{partialIrred->ACIred}). Consider for instance the two equivalent ACI-matrices $E'=\begin{smat} x \\ 1+x \end{smat}\sim F'=\begin{smat} x \\ 1 \end{smat}$ over any field: $E'$ is irreducible of constant rank one, $F'$ is equivalent to $E'$, and $F'$ is row reducible.} 

As we explained in Section~\ref{equivACIs},  equivalent ACI-matrices represent the same geometrical collection of objects. So it would make sense to have a stronger concept of irreducibility, one that is preserved by equivalence. This motivates the following definition.

\begin{definition} \label{completelyIrreducible}
Let $A$ be a $m\times n$ an ACI-matrix of constant rank $\rho$ over a field $\mathbb{F}$. We  say that
 $A$  is {\bf completely irreducible} if each  ACI-matrix equivalent to $A$ is irreducible.
 \end{definition}
 
We have studied the effect on the rank of an ACI-matrix of constant rank when we delete one of its columns (or one of its rows). We are also interested in the effect on the rank  of an ACI-matrix of constant rank when we add one constant column. 

\begin{definition} \label{columnaugmentable}
Let $A$ be a $m\times n$  ACI-matrix of constant rank $\rho$ over a field $\mathbb{F}$. We say that $A$ is {\bf column augmentable} if there exists some   $v\in \mathbb{F}^m$ such that the augmented ACI-matrix $\big[\, A \ v\big]$ is of constant rank $\rho+1$. Otherwise we will say that $A$ is {\bf column non-augmentable}
 \end{definition}

\begin{remark}\label{mfrMfrCI}
Let $A$ be a $m\times n$ ACI-matrix of constant rank $\rho$. From Definitions~\ref{mfr},~\ref{irreducible} and~\ref{columnaugmentable}  it follows that:
\begin{enumerate}[(1)]
\item $A$  is minimal full rank if and only if $\rho=m<n$ and $A$ is column irreducible. 
\item $A$  is maximal full rank if and only if $\rho=n<m$ and $A$ is column non-augmentable.
\end{enumerate}
\end{remark}

In the next result we will see how we can study the complete irreducibility of an ACI-matrix without considering all its equivalent ACI-matrices. 
 
 \begin{theorem} \label{alternativeDefOfCI}
The  $m\times n$ ACI-matrix $A$ of  constant rank $\rho$ is completely irreducible if and only if:
 \begin{enumerate}[(a)]
\item  $A$ is column irreducible. 
\item  $A$ is column non-augmentable.
\end{enumerate}
\end{theorem}
\begin{proof}
That $A$ is completely irreducible means that $TAQ$ is irreducible for any nonsingular constant  $T$ of order $m$ and any permutation  $Q$ of order $n$ or, equivalently, that  $TA$ is irreducible for any nonsingular constant  $T$ of order $m$. In turn, this is equal to say that $TA$ is column irreducible and row irreducible for any nonsingular constant  $T$ of order $m$.  And observe that  $TA$ is column irreducible if and only if $A$ is column irreducible because
$$\rank(C_j(TA))=\rank(TC_j(A))=\rank(C_j(A))$$ 
where   $C_j(A)$ and $C_j(TA)$ denote the ACI-matrices obtained by deleting the column $j$ of  $A$ and $TA$.    

In summary, $A$ is completely irreducible if and only if  $A$ is column irreducible and $TA$ is row irreducible for any nonsingular constant $T$ of order $m$. Observe that we finish the proof of our theorem if we  prove that: 
\begin{center}
$TA$ is row irreducible for any nonsingular constant $T$ $\Longleftrightarrow$ $A$ is column non-augmentable. 
\end{center}
Actually we will prove the opposite affirmation:
\begin{center}
 $TA$ is row reducible for some nonsingular constant $T$ $\Longleftrightarrow$ $A$ is column augmentable. 
\end{center}

\begin{description}
\item[$\Rightarrow$)] Let $T$ be a nonsingular constant matrix of order $m$ such that $TA$ be row reducible. Then there is an $i\in \{1,\ldots, m\}$ such that if we delete the $i$-th row from $TA$ the resulting ACI-matrix remains of constant rank $\rho$.   Without loss of generality assume that $i=1$.  Let $R_1(TA)$ be the ACI-matrix that we obtain  deleting the first row of $TA$. Then 
$$
\rank [TA | e_1]=\rank \left[\begin{array}{c|c}  TA & {\scriptsize \begin{array}{c} 1 \\ 0 \\ \vdots \\ 0 \end{array} } \end{array}\right]= 
\rank{\scriptsize \left[\begin{array}{c|c}  * & 1 \\ \hline \multicolumn{1}{c|}{\multirow{3}{*}{\large$R_1(TA)$}} & 0 \\  &\vdots \\ & 0  \end{array}\right] }=\{\rho+1\}
$$
so 
\[
\rank[A | T^{-1}e_1 ]=\rank(T [A | T^{-1}e_1])=\rank[TA | e_1]=\{\rho+1\}.
\]
Therefore $A$ is column augmentable with vector $T^{-1}e_1$.

\item[$\Leftarrow$)] Let $v$ be a non-zero constant vector for which $\big[A | \ v\big]$ has constant rank $\rho+1$. Let $T$ be a nonsingular constant matrix of order $m$ such that $Tv=e_1$. Then
\[
\rank{\scriptsize \left[\begin{array}{c|c}  * & 1 \\ \hline \multicolumn{1}{c|}{\multirow{3}{*}{\large$R_1(TA)$}} & 0 \\  &\vdots \\ & 0  \end{array}\right] } =\rank[TA | e_1 ]=\rank(T^{-1}[TA | e_1])=\rank[A | v]=\{\rho+1\}.
\]
So $R_1(TA)$ has constant rank $\rho$, and so $TA$ is row reducible because of its first row.
\end{description}
\end{proof}

\section{Completely irreducible ACI-matrices
} \label{completelyirreduciblefullrank}

The previous section should have convinced us that completely irreducible ACI-matrices deserve to be analyzed and fully understood. We will first analyze completely irreducible ACI-matrices which are full rank, after that those which are not full rank. Then 
we will make some remarks on how to construct completely irreducible ACI-matrices.
And finally we will establish where do the completely irreducible ACI-matrices appear in Theorem~\ref{ResultadosAnteriores3}. 

\subsection{Completely irreducible  ACI-matrices  which are full rank}\label{CIfullrank}

In the next result we will show that  the concept of complete irreducibility  for full rank ACI-matrices encompasses the concepts of square, minimal and maximal full rank.

\begin{proposition} \label{SFC}
Let $A$ be an $m\times n$ ACI-matrix of constant rank $\rho$. Then 
\begin{enumerate}[(i)]
\item  $A$ is completely irreducible of constant rank $\rho=m=n$ if and only if $A$ is square full rank.
\item $A$ is  completely irreducible of constant rank $\rho=m<n$   if and only if $A$ is  minimal full rank.
\item $A$ is  completely irreducible of constant rank $\rho=n<m$ if and only if $A$ is maximal full rank. 
\end{enumerate}
\end{proposition}

\begin{proof}  
\begin{enumerate}[$(i)$]
\item The necessary part  is trivial. The sufficient part    is based in two clear facts: that a square full rank ACI-matrix is irreducible, and that  the ACI-matrices which are equivalent to a square full rank ACI-matrix are square full rank. 
\item $\Rightarrow)$ Assume that $A$ is  completely irreducible of constant rank $\rho=m<n$.   By item (a) of Theorem~\ref{alternativeDefOfCI}, $A$ is column irreducible. So, by item (1) of Remark~\ref{mfrMfrCI}, $A$ is minimal full rank.

$\Leftarrow)$ Assume that $A$ is  minimal full rank. So $A$ has constant rank $m<n$ and therefore $A$ is column non-augmentable, since for each $v\in \mathbb{F}^m$ the augmented matrix $\big[A | \ v\big]$ has size $m\times (n+1)$ and  constant rank $m$. On the other hand, by item (1) of Remark~\ref{mfrMfrCI}, $A$ is column irreducible. So, by  Theorem~\ref{alternativeDefOfCI} $A$ is completely irreducible of constant rank $m<n$.

\item $\Rightarrow)$ Assume that $A$ is  completely irreducible of constant rank $\rho=n<m$.   By item (b) of Theorem~\ref{alternativeDefOfCI}, $A$ is column non-augmentable. So, by item (2) of Remark~\ref{mfrMfrCI}, $A$ is maximal full rank.

$\Leftarrow)$ Assume that $A$ is  maximal full rank. So $A$ has constant rank $n<m$ and therefore $A$ is column irreducible, since if we delete one column of $A$ we obtain an ACI-matrix of size $m\times (n-1)$ and  constant rank $n-1$. On the other hand, by item (2) of Remark~\ref{mfrMfrCI}, $A$ is column non-augmentable. So, by  Theorem~\ref{alternativeDefOfCI} $A$ is completely irreducible of constant rank $n<m$.

\end{enumerate}
\end{proof}

\subsection{Completely irreducible  ACI-matrices  which are not full rank}

In our next result we   characterize  the completely irreducible ACI-matrices which are not full rank.

\begin{theorem} \label{FixedRankIntoCD2}
Let $A$ be a $m \times n$     ACI-matrix over a field $\mathbb{F}$ and let $\rho$ an integer with $1\leq \rho < \min \{m, n\}$. Then $A$ is completely irreducible  of constant rank $\rho$ if and only if for some positive integers $\rho_1$ and $\rho_2$ such that  $\rho_1+\rho_2=\rho$ we have that  $ A\sim   \begin{smat} A_{11} & *  \\   0 & A_{22} \end{smat} $ where  $A_{11}$ is minimal full rank of constant rank $\rho_1$ and   $A_{22}$ is maximal full rank of constant rank $\rho_2$.
\end{theorem}

\begin{proof} Let $A$ be a $m \times n$     ACI-matrix  and let $\rho$ an integer with $1\leq \rho < \min \{m, n\}$. 

\begin{itemize}
\item[$\Rightarrow$)] As $A$ has constant rank $\rho$ with $1\leq \rho < \min \{m, n\}$ then we can apply to $A$ the item $(iv)$ of  Theorem~\ref{ResultadosAnteriores3}. Moreover, by hypothesis  $A$  is completely irreducible and so  we must apply exactly the case $(a)$ of item $(iv)$  of  Theorem~\ref{ResultadosAnteriores3}. So, for some positive integers $r<m$ and $s<n$ with $\rho=(m-r)+(n-s)$ there exist a nonsingular constant  matrix $T$ of order $m$ and a permutation matrix $Q$ of order $n$ such that 
  \begin{equation*}
{\small TAQ = \left[\begin{array}{cc|c}
B & * & *  \\ \hline
\multicolumn{2}{c|}{\multirow{2}{*}{$0_{r\times s}$}} & * \\ 
 & & C
\end{array}\right] }
\end{equation*}
with either $ B= { \begin{smat} 
 1 &  & *  \\ 
  & \ddots &   \\ 
 0 & & 1  
\end{smat}}$ or  $B$ minimal full rank, and with either $C = { \begin{smat} 
 1 &  & *  \\ 
  & \ddots &   \\ 
 0 & & 1  
\end{smat}}$ or $C$  maximal full rank. Let $\rho_1=m-r$ and $\rho_2=n-s$. Note that $\rho_1>0$, 
that $\rho_2>0$, 
and that $\rho_1+\rho_2=\rho$. As $B$ has $\rho_1$ rows and $C$ has $\rho_2$ columns then  it follows from Lemma~\ref{SC+FRC=constant-rank} that 
$$\rank(A)=\rank(TAQ)=\{\rho_1+\rho_2\}=\{\rho\}.$$
As $A$ is completely irreducible then $TAQ$ is  irreducible and so
\begin{equation*}
{\small  \left[\begin{array}{cc|c}
B & * & *  \\ \hline
\multicolumn{2}{c|}{\multirow{2}{*}{$0_{r\times s}$}} & * \\ 
 & & C
\end{array}\right]= \left[\begin{array}{c|c}
B &  *  \\ \hline
 0_{r\times s} & C
\end{array}\right], }
\end{equation*}
otherwise we could delete one row or one column without changing the rank of $TAQ$. Note that if $ B= { \begin{smat} 
 1 &  & *  \\ 
  & \ddots &   \\ 
 0 & & 1  
\end{smat}}$ then $\rank(A)=\{n\}$ and that if $ C= { \begin{smat} 
 1 &  & *  \\ 
  & \ddots &   \\ 
 0 & & 1  
\end{smat}}$ then $\rank(A)=\{m\}$. None of both possibilities are valid  since $A$ has constant rank $\rho <   \min \{m, n\}$. Then $B$ is minimal full rank of constant rank $\rho_1$ and $C$ is maximal full rank of constant rank $\rho_2$.

\item[$\Leftarrow$)] 
As complete irreducibility is preserved by equivalence then, without loss of generality, we can assume that  $A= \begin{smat} A_{11} &  A_{12}  \\  0 & A_{22} \end{smat}$ where $A_{11}$ is  $\rho_1\times (n-\rho_2)$ minimal full rank of constant rank $\rho_1$ and $A_{22}$ is  $(m-\rho_1)\times \rho_2$ maximal full rank  of constant rank $\rho_2$. From Theorem~\ref{alternativeDefOfCI}, we need to prove that:

\begin{itemize}
\item \textbf{$A$ is column irreducible}. 

We  consider two cases. In both cases we will use the notation  $C_k(H)$ for  the ACI-matrix obtained by deleting the column $k$ of the ACI-matrix $H$.   
\begin{itemize}
\item[1)]  $j\in \{1,\ldots,n-\rho_2\}$. As $A_{11}$ is minimal full rank then $\min\{\rank(C_j(A_{11}))\}<\rho_1$. So 
{\small
$$ \hspace{-8mm}
\min\big\{ \rank(C_j(A))\big\} = \min\big\{ \rank \begin{bmatrix}
C_j(A_{11}) & A_{12}\\
0 & A_{22}
\end{bmatrix} \big\}= \min\big\{\rank \begin{bmatrix}
C_j(A_{11}) \\
0 
\end{bmatrix}\big\} +\rho_2 =(\rho_1-1)+\rho_2=\rho-1.
$$}
\item[2)]  $j\in \{n-\rho_2+1,\ldots,n\}$. As $A_{22}$ is maximal full rank then $\rank\big(C_{j-(n-\rho_2)}(A_{22})\big)=\{\rho_2-1\}$. So we can apply Lemma~\ref{SC+FRC=constant-rank} to conclude that 
{\small $$
 \rank\big(C_j(A)\big) =  \rank \begin{bmatrix}
A_{11} & C_{j-(n-\rho_2)}(A_{12})\\
0 & C_{j-(n-\rho_2)}(A_{22})
\end{bmatrix}  = \{\rho_1 +(\rho_2-1)\}=\{\rho-1\}. 
$$}
\end{itemize}
From 1) and 2) we conclude that  $A$ is column irreducible.  

\item\textbf{$\scriptsize\begin{bmatrix}[cc|c] A_{11} &  A_{12} & u_1  \\  0 & A_{22} & u_2 \end{bmatrix}$ is never of constant rank $\rho+1$  for $\scriptsize\begin{bmatrix} u_1 \\ u_2 \end{bmatrix} \in \mathbb{F}^m$}.  

Since $A_{22}$ is maximal full rank then $\begin{bmatrix}[c|c] A_{22} & u_2\end{bmatrix}$ is not full rank. So $\begin{bmatrix}[c|c] A_{22} & u_2\end{bmatrix}$ has a completion $\renewcommand\arraystretch{1.1}\begin{bmatrix}[c|c]\widehat{A_{22}} & u_2\end{bmatrix}$ for which there exists a nonzero constant vector $\begin{bmatrix}w\\ \lambda \end{bmatrix} \in \mathbb{F}^{\rho_2+1}$  with $\lambda\neq 0$ such that: 
$$\renewcommand\arraystretch{1.1}\begin{bmatrix}[c|c]\widehat{A_{22}} & u_2\end{bmatrix} \begin{bmatrix}w\\ \lambda \end{bmatrix}=\mathbf{0} \in \mathbb{F}^{m-\rho_1}.$$ 
As $\begin{bmatrix}[c|c] A_{22} & u_2\end{bmatrix}$ and $\begin{bmatrix}[c|c] A_{12} & u_1\end{bmatrix}$ may share  variables then the completion of $\begin{bmatrix}[c|c] A_{22} & u_2\end{bmatrix}$ may force a partial completion of $\begin{bmatrix}[c|c] A_{12} & u_1\end{bmatrix}$ which we fully complete in any way we want to $\renewcommand\arraystretch{1.1}\begin{bmatrix}[c|c] \widehat{A_{12}} & u_1\end{bmatrix}$. Define the constant  vector  
\[a := \renewcommand\arraystretch{1.1}\begin{bmatrix}[c|c] \widehat{A_{12}} & u_1\end{bmatrix} \begin{bmatrix}w\\ \lambda \end{bmatrix}\in \mathbb{F}^{\rho_1}.
\]

Since $A_{11}$ is full rank, then for any completion $\widehat{A_{11}}$ there exists a constant vector $v \in \mathbb{F}^{n-\rho_2}$ such that $\widehat{A_{11}}\; v = -a.$

Finally,
\begin{align} \label{adjoin-Acolumn}
\renewcommand\arraystretch{1.5}\begin{bmatrix}[cc|c]
\widehat{A_{11}} & \widehat{A_{12}} & u_1\\
0    & \widehat{A_{22}} & u_2
\end{bmatrix}
\begin{bmatrix}
v \\
w \\
\lambda
\end{bmatrix} = 
\begin{bmatrix}
-a +a \\
\mathbf{0}
\end{bmatrix} = 
\begin{bmatrix}
\mathbf{0} \\
\mathbf{0}
\end{bmatrix}\in \mathbb{F}^{m}.
\end{align}
Note that  $\lambda\neq {0}$. 
Therefore $\scriptsize\renewcommand\arraystretch{1.5}\begin{bmatrix}[ccc]
 u_1\\
 u_2
\end{bmatrix}$ depends linearly of the columns of $\scriptsize\renewcommand\arraystretch{1.5}\begin{bmatrix}[ccc]
\widehat{A_{11}} & \widehat{A_{12}} \\
0    & \widehat{A_{22}} 
\end{bmatrix}$, thus  the completion $\scriptsize\begin{bmatrix}[cc|c] \widehat{A_{11}} & \widehat{A_{12}} & u_1  \\  0 & \widehat{A_{22}} & u_2 \end{bmatrix}$ has rank $\rho$.
 \end{itemize}
 \end{itemize}
\end{proof}

\subsection{Constructing Completely Irreducible ACI-matrices}\label{}

In Theorem~\ref{FixedRankIntoCD2} we have seen that if $A=\begin{smat} A_{11} & * \\ 0 & A_{22} \end{smat}$ where $A_{11}$ is minimal full rank and $A_{22}$ is maximal full rank then $A$ is completely irreducible. This takes us to consider the question of whether we can use the completely irreducible ACI-matrices which are full rank (square, minimal and maximal) as building blocks to construct new completely irreducible ACI-matrices. 

Namely, the question can be stated in the following terms: Is the ACI-matrix 
\[
A=\begin{bmatrix} A_{11} &  A_{12}  \\  0 & A_{22} \end{bmatrix}
\] 
 completely irreducible where $A_{11}$ and $A_{22}$ are either square, minimal or maximal full rank? We have nine different cases:

\begin{enumerate}[$(i)$]

\item If $A_{11}$  and $A_{22}$ are square full rank, then $A$ is completely irreducible.

Clearly $A$ is square full rank and, by Proposition~\ref{SFC}, it is completely irreducible.

\item If $A_{11}$ is square full rank and $A_{22}$ is minimal full rank, then $A$ is not always  completely irreducible. 

Consider the ACI-matrix over $\mathbb{F}_2$
\[
A=\begin{bmatrix} A_{11} &  A_{12}  \\  0 & A_{22} \end{bmatrix}=\left[
\begin{array}{cc|ccccc}
 1 & x_1 & 0 & 0 & 0 & 0 & 0 \\
 0 & 1 & y_1 & 0 & 0 & 1 & 0 \\ \hline
 0 & 0 & 1 & y_2 & y_3 & 0 & 0 \\
 0 & 0 & 0 & 0 & y_3 & y_4 & 1 \\
 0 & 0 & y_1 & 1 & 1 & 1 & y_5 \\
\end{array}
\right].
\]
It satisfies  the following facts:
\begin{enumerate}
\item  $A_{11}$ is square full rank.

\item $A_{22}$ is minimal full rank. First, that $A_{22}$ is full rank can be checked  directly since it has 5 variables and admits $32=2^5$ different completions, all of them of rank 3. Moreover,  if we delete any  of its five columns  the resulting $3\times 4$ ACI-matrix is not full rank since it admits a completion of rank 2. 

\item $A$ is full rank. It follows   from the structure of $A$  since any completion of  $A_{11}$ has rank 2 and any completion of $A_{22}$ has rank 3, so  any completion of $A$ has rank 5.
\item $A$ is column reducible. If we delete the second column of $A$ then we obtain an ACI-matrix $A'$ that  has 5 variables and  admits $32=2^5$ different completions, all of them of rank 5. So $A'$ is full rank, which implies that  $A$ is column reducible.
\end{enumerate}
Since $A$ is column reducible then $A$ is not irreducible, and thus $A$ is not completely irreducible.

\item If $A_{11}$ is square full rank and $A_{22}$ is maximal full rank, then $A$ is completely irreducible. 

That $A$ is maximal full rank was proved in~\cite[Lemma 2.2]{BoCa2}, and that $A$ is completely irreducible is a consequence of  Proposition~\ref{SFC}.

\item If $A_{11}$ is minimal full rank and $A_{22}$ is square full rank, then $A$ is completely irreducible. 

Consider the $m\times n$ ACI-matrix
\[
A=\begin{bmatrix} A_{11} &  A_{12}  \\  0 & A_{22} \end{bmatrix}
\] 
where $A_{11}$ is $\rho_1\times\theta_1$ minimal full rank (so it has constant rank $\rho_1$ and $\rho_1<\theta_1$) and $A_{22}$ is $\rho_2\times\rho_2$ square full rank (so it has constant rank $\rho_2$).  By Applying Lemma~\ref{SC+FRC=constant-rank} we have  that $A$ is full rank of constant rank $\rho_1+\rho_2=m$. 
Observe also that $m=\rho_1+\rho_2<\theta_1+\rho_2=n$. 

According to Proposition~\ref{SFC} the result follows if we prove that $A$ is minimal full rank. And in turn, according to Remark~\ref{mfrMfrCI} $(1)$, it is enough to show that the ACI-matrix obtained by deleting any column of $A$ admits a completion of rank $m-1$:
\begin{itemize}
\item If we delete one of its first $\theta_1$ columns then the resulting ACI-matrix is of type $\begin{smat} {A'_{11}} & {A_{12}} \\ 0& {A_{12}} \end{smat}$.  As $A_{11}$ is minimal full rank, then  there exists a completion  $\widehat{A'_{11}}$ of $A'_{11}$ whose rank is $\rho_1-1$. Extend this completion so that  $\begin{smat} \widehat{A'_{11}} & \widehat{A_{12}} \\ 0& \widehat{A_{12}} \end{smat}$ is a completion of $\begin{smat} {A'_{11}} & {A_{12}} \\ 0& {A_{12}} \end{smat}$. Then 
$$\rank \begin{bmatrix} \widehat{A'_{11}} & \widehat{A_{12}} \\ 0& \widehat{A_{12}} \end{bmatrix} = 
\rank \begin{bmatrix} \widehat{A'_{11}} & 0 \\ 0& \widehat{A_{12}} \end{bmatrix} =
\rank (\widehat{A'_{11}})   +
\rank (\widehat{A_{22}}) =(\rho_1-1)+\rho_2=m-1$$
\item If we delete one of its last $\rho_2$ columns then the resulting ACI-matrix is of type $\begin{smat} A_{11} & A'_{12}\\ 0 & A'_{22}\end{smat}$. The $\rho_2\times(\rho_2-1)$ ACI-matrix $A'_{22}$ has constant rank $\rho_2-1$.  By Lemma~\ref{SC+FRC=constant-rank} we have that $\begin{smat} A_{11} & A'_{12}\\ 0 & A'_{22}\end{smat}$ has constant rank $\rho_1+(\rho_2-1)=m-1$.
\end{itemize}

\item If $A_{11}$ and $A_{22}$ are minimal full rank, then $A$ is not always  completely irreducible.

Consider the ACI-matrix over $\mathbb{F}_2$
$$
A=\begin{bmatrix} A_{11} &  A_{12}  \\  0 & A_{22} \end{bmatrix}=
\left[
\begin{array}{ccc|ccccc}
x_1 & 1 & x_3 & 0 & 0 & 0 & 0 & 0 \\
x_1+1 & x_2 & 1 & y_1 & 0 & 0 & 1 & 0 \\ \hline
 0 & 0 & 0 & 1 & y_2 & y_3 & 0 & 0 \\
 0 & 0 & 0 & 0 & 0 & y_3 & y_4 & 1 \\
 0 & 0 & 0 & y_1 & 1 & 1 & 1 & y_5 \\
\end{array}
\right]
$$
We can check that $A$ is full rank since any completion has rank 5, that $A_{11}$ and $A_{22}$ are minimal full rank, and that $A$ is column reducible since the ACI-matrix obtained after deleting the first column of $A$ has constant rank 5. The  checks are similar to those on item $(ii)$. So $A$ is not completely irreducible.

\item If $A_{11}$ is minimal full rank and $A_{22}$ is maximal full rank, then $A$ is completely irreducible.

This is part  of  Theorem~\ref{FixedRankIntoCD2} and it is  quite surprising. Observe that Theorem~\ref{FixedRankIntoCD2} tell us that, up to equivalence, every completely irreducible ACI-matrices which is non-full rank can be constructed in this way. 

\item If $A_{11}$ is maximal full rank and $A_{22}$ is square full rank,  then $A$ is completely irreducible. 

In this case the position of $A_{11}$ and $A_{22}$ is interchanged with respect to the case $(iii)$ above. We can adapt easily the proof in~\cite[Lemma 2.2]{BoCa2} to show that $A$  is maximal full rank. And that $A$ is completely irreducible is a consequence of  Proposition~\ref{SFC}.

\item If $A_{11}$ is maximal full rank and $A_{22}$ is minimal full rank, then  $A$ is not always completely irreducible. 

Consider the ACI-matrix over $\mathbb{F}_2$
\[
A=\begin{bmatrix} A_{11} &  A_{12}  \\  0 & A_{22} \end{bmatrix}=\left[
\begin{array}{ccc|ccc}
 1 & 1 & 1 & 0 & 0 & 0 \\
 0 & 1 & x_1 & 0 & 0 & 0 \\
 x_2 & 0 & 1 & 0 & 0 & 0 \\
 0 & x_3 & 1 & 1 & 0 & 0 \\ \hline
 0 & 0 & 0 & y_1 & 1 & y_3 \\
 0 & 0 & 0 & y_1+1 & y_2& 1 \\
\end{array}
\right]
\]
It can be checked that  $A_{11}$ is maximal full rank,  that $A_{22}$ is minimal full rank, and that    some completions of $A$ have rank $5$ and  other completions of $A$ have rank $6$. Then $A$ has no  constant rank. So it can not be completely irreducible. 

\item If $A_{11}$ and $A_{22}$ are maximal full rank, then $A$ is completely irreducible. 

That $A$ is maximal full rank was proved  in~\cite[Lemma 2.3]{BoCa2}, and that $A$ is completely irreducible is a consequence of Proposition~\ref{SFC}.

\end{enumerate}

\noindent The following table summarizes all the possibilities:
\bigskip
\begin{equation}\label{Tabla}
\smash{\overset{A_{11}}{\overbrace{
    \begin{array}{| l | l | l | l } 
    \multicolumn{1}{c}{\text{Square FR}} & \multicolumn{1}{c}{\text{Minimal FR}} & \multicolumn{1}{c}{\text{Maximal FR}}  \\ \cline{1-3}
    \text{$(i)$ C.I. } & \text{$(iv)$ C.I.} & \text{$(vii)$  C.I.}  \\ \cline{1-3}
    \text{$(ii)$  Not always C.I.} & \text{$(v)$ Not always C.I.} & \text{$(viii)$ Not always C.I.}  \\ \cline{1-3}
    \text{$(iii)$ C.I. } & \text{$(vi)$ C.I.} & \text{$(ix)$ C.I. }  \\ \cline{1-3}
    \end{array}
}}}    
\begin{array}{l} 
      \\
\left. \begin{array}{l} 
    \text{Square FR}   \\ 
    \text{Minimal FR}  \\ 
    \text{Maximal FR}  \\ 
  \end{array} \right\} A_{22}
\end{array}
\end{equation}
\bigskip

\section{The core of a constant rank ACI-matrix}\label{PartialMatricesOfConstantRank}

%A {\bf partial matrix} is a matrix such that each one of its entries is a constant or an indeterminate and the indeterminates only appear once. Partial matrices are a subclass of ACI-matrices, so the definition of rank for partial matrices is inherited from the one for ACI-matrices.  McTigue and Quinlan studied the rank of partial matrices in~\cite{McTigueQuinlan,McTigueQuinlan2,McTigueQuinlan3}.
%
%
%If A is a $m\times n$ constant matrix of rank $\rho$ then it is well known that we can delete $m-\rho$ rows and $n-\rho$ columns in such a way that the $\rho \times \rho$ submatrix of $A$ that we obtain has rank $\rho$.  We might naively expect that any partial matrix  with constant rank $\rho$ must also have a $\rho\times \rho$ submatrix of constant rank $\rho$. But McTigue and Quinlan proved  that this is not the case. 
%
%\begin{theorem}(\cite{McTigueQuinlan3})\label{PartialSubmatrices}
%Every partial matrix $A$ of constant rank $\rho$ over a field $\mathbb{F}$ possesses an $\rho\times \rho$  submatrix of constant rank $\rho$ if and only if $|\mathbb{F}|\geq \rho$. 
%\end{theorem}
%
%Theorem~\ref{PartialSubmatrices} says  that not all partial matrices of constant rank $\rho$ have  a submatrix of size  $\rho\times \rho$ of constant rank $\rho$. Indeed, McTigue and Quinlan showed that if $|\mathbb{F}|< \rho$ then there exist examples of partial matrices of size $m\times n$ with $\max\{m,n\}\geq \rho+|\mathbb{F}|-1$ that does not contain a $\rho\times \rho$ submatrix with rank $\rho$ (see for instance the partial matrix $A_q$ for $q=2$ given in~(\ref{thePartialMFR})). 
%
We know that not all ACI-matrices of constant rank $\rho$ have a $\rho\times \rho$ submatrix of constant rank $\rho$. 
It could be expected  that each  ACI-matrix   of constant rank $\rho$ has at least  a submatrix of constant rank $\rho$ that is completely irreducible. But again this is not the case. Consider the $7\times 5$ partial matrix $E$  given in~(\ref{partialIrred->ACIred}): the unique submatrix of $E$ of constant rank 5 is $E$, and $E$ is not completely irreducible  since $E\sim F$ and $F$  is row reducible. 

So not all ACI-matrices of  constant rank $\rho$ have  a submatrix of constant rank $\rho$ that is completely irreducible. This removes one tool that could be employed in the calculus of  the rank of an ACI-matrix.  We can in some way offset this situation.

\begin{definition} 
Let $A$ be an ACI-matrix of constant rank $\rho$ over a field $\mathbb{F}$. If 
\begin{align*}
A\sim  \begin{bmatrix}[cccccc]
 A' & * \\
 *   & * \\ 
\end{bmatrix}
\end{align*}
where $A'$ is  completely irreducible of constant rank $\rho$ then  $A'$ is said to be a \bf{core} of  $A$.
\end{definition}

So, complete irreducibility allows  to generalize, in some way,   the concept of $\rho\times \rho$ submatrix of rank $\rho$. Given an ACI-matrix $A$ with constant rank $\rho$ it seems like a good idea to find a representative of its equivalence class that verifies that: it has a completely irreducible ACI-submatrix of rank $\rho$ (a core), and has a structure which is simple and makes it clear why it is of constant rank $\rho$. Theorem~\ref{ResultadosAnteriores3} finds such a representative, as we will see in the proof of the next result.

\begin{lemma}\label{Core}
Any  ACI-matrix of constant rank over a field  has a core.
\end{lemma}

\begin{proof} Let  $A$ be a  $m\times n$ ACI-matrix with  constant rank $\rho$.  We have several possibilities: 
\begin{enumerate}[$(i)$]

\item If $\rho = m=n$ then, by Propositition~\ref{SFC}, $A$ is completely irreducible.  So $A$ is a core of $A$. 

We can find a core of $A$ with a simpler structure. According to  Theorem~\ref{ResultadosAnteriores3}, $A\sim \begin{smat} \vspace{-2mm} 1 & & * \\   & \ddots &  \\ 0 & & 1 \end{smat}$.  So $\begin{smat} \vspace{-2mm} 1 & & * \\   & \ddots &  \\ 0 & & 1 \end{smat}$, which is  completely irreducible,  is also a core of $A$. 

\item  If $ \rho = m<n$ then, according to  Theorem~\ref{ResultadosAnteriores3}, $A \sim \small{\begin{bmatrix}  B & * \, \end{bmatrix}}$ where $B=\begin{smat} \vspace{-2mm} 1 & & * \\   & \ddots &  \\ 0 & & 1 \end{smat}$ of size $m\times m$ or $B$ is  minimal full rank of size $m\times n'$ with $m<n'\leq n$. In any case $B$ is completely irreducible with constan rank $\rho$ (see Proposition~\ref{SFC}). So $B$ is a core of $A$. 

\item  If $ \rho = n<m$ then, according to  Theorem~\ref{ResultadosAnteriores3}, $A \sim {\begin{smat}  * \\ \\   C   \end{smat}}$ where $C=\begin{smat} \vspace{-2mm} 1 & & * \\   & \ddots &  \\ 0 & & 1 \end{smat}$ of size $n\times n$ or $C$ is  maximal full rank of size $m'\times n$ with $n<m'\leq m$. In any case $C$ is completely irreducible with constant rank $\rho$ (see Proposition~\ref{SFC}). So $C$ is a core of $A$.

\item If  $ \rho <  \min \{m, n\}$   then, according to  Theorem~\ref{ResultadosAnteriores3},  we have three possibilities:
\begin{enumerate}[$(a)$]
\item For some positive integers $r<m$ and $s<n$ with $\rho=(m-r)+(n-s)$ we have
$$A\sim \small{\left[\begin{array}{cc|c}
B & * & *  \\ \hline
\multicolumn{2}{c|}{\multirow{2}{*}{$0_{r\times s}$}} & * \\ 
 & & C
\end{array}\right]} 
$$
where  $B=\begin{smat} \vspace{-2mm} 1 & & * \\   & \ddots &  \\ 0 & & 1 \end{smat}$ or $B$ is minimal full rank  and  $C=\begin{smat} \vspace{-2mm} 1 & & * \\   & \ddots &  \\ 0 & & 1 \end{smat}$ or $C$ is maximal full rank. By permuting some rows and some columns of the last ACI-matrix we have that 
$$ \small{\left[\begin{array}{cc|c}
B & * & *  \\ \hline
\multicolumn{2}{c|}{\multirow{2}{*}{$0_{r\times s}$}} & * \\ 
 & & C
\end{array}\right]} \sim  
\small{\left[\begin{array}{cc|c}
B & * & *  \\ 
0 & C & 0 \\ \hline
0 & * & 0
\end{array}\right]}
$$
According to the table given in~(\ref{Tabla}) the ACI-matrix  $\begin{smat}B & * \\ 0 & C \end{smat}$ is completely irreducible (the possible cases  are those corresponding to items $(i), (iii), (iv)$ or $(vi)$)  with constant rank $(m-r)+(n-s)$ (see  Lemma~\ref{SC+FRC=constant-rank}), then it  is a core of $A$.

\item For some positive integer $r<m$  with $\rho=m-r$ we have
$A\sim \small{\left[\begin{array}{cc}
B & *   \\ \hline
\multicolumn{2}{c|}{\multirow{1}{*}{$0_{r\times n}$}}  
\end{array}\right]} 
$
where  $B=\begin{smat} \vspace{-2mm} 1 & & * \\   & \ddots &  \\ 0 & & 1 \end{smat}$ or $B$ is minimal full rank.  As  $B$ is completely irreducible  with constant rank $m-r$ then $B$  is a core of $A$.

\item For some positive integer  $s<n$ with $\rho=n-s$ we have
$A\sim \small{\left[\begin{array}{cc|c}
\multicolumn{2}{c|}{\multirow{2}{*}{$0_{m\times s}$}} & * \\ 
 & & C
\end{array}\right]} 
$
where   $C=\begin{smat} \vspace{-2mm} 1 & & * \\   & \ddots &  \\ 0 & & 1 \end{smat}$ or $C$ is maximal full rank. As $C$  is completely irreducible with constant rank $n-s$  then $C$  is a core of $A$.
 
\end{enumerate}
\end{enumerate}
\end{proof}

Observe that Lemma~\ref{Core} for  ACI-matrices of constant rank has  analogy  with  Theorem~\ref{PartialSubmatrices}  for partial matrices of constant rank.

\subsubsection*{Acknowledgements}
We thank the referee, his comments and questions have prompted us to carry out a careful revision of a previous version of this paper. 

\bibliographystyle{plain}

\begin{thebibliography}{1}

\bibitem{BoCa1}
A. Borobia, R. Canogar,
\newblock Nonsingular ACI-matrices over integral domains.
\newblock {\em Linear Algebra Appl.}, 436:4311--4316, 2012.

\bibitem{BoCa2}
A. Borobia, R. Canogar,
\newblock Characterization of full rank ACI-matrices over fields.
\newblock {\em Linear Algebra Appl.}, 439:3752--3762, 2013.

\bibitem{MR2680270}
R. Brualdi, Z. Huang, X. Zhan,
\newblock Singular, nonsingular, and bounded rank completions of ACI-matrices.
\newblock {\em Linear Algebra Appl.}, 433:1452--1462, 2010.

%\bibitem{CJRW}
%N. Cohen, C.R. Johnson, L. Rodman, H. Woerdeman, 
%\newblock Ranks of completions of partial matrices.
%\newblock {\em Oper. Theory Adv. Appl.}, 40:16--185,1989.


\bibitem{MR2775784}
Z. Huang, X. Zhan,
\newblock ACI-matrices all of whose completions have the same rank,
\newblock {\em Linear Algebra Appl.}, 434:1956--1967, 2011.

\bibitem{McTigueTesis}
J. McTigue, 
\newblock Completion of partial matrices.
\newblock {\em Thesis}, The National University of Ireland, 2015.


\bibitem{McTigueQuinlan}
J. McTigue, R. Quinlan,
\newblock Partial matrices whose completions have ranks bounded below.
\newblock {\em Linear Algebra Appl.}, 435:2259--2271, 2011.

\bibitem{McTigueQuinlan2}
J. McTigue, R. Quinlan,
\newblock Partial matrices whose completions all have the same rank.
\newblock {\em Linear Algebra Appl.}, 438:348--360, 2013.

\bibitem{McTigueQuinlan3}
J. McTigue, R. Quinlan,
\newblock Partial matrices of constant rank.
\newblock {\em Linear Algebra Appl.}, 446:177--191, 2014.

\end{thebibliography}

\end{document}